\documentclass[11pt]{amsart}

\numberwithin{equation}{section} 
\usepackage{amssymb}
\usepackage[utf8]{inputenc}
\usepackage[T1]{fontenc}
\usepackage[letterpaper,top=1in, left=1in, right=1in, bottom=1in]{geometry}
\usepackage{color,amsmath,mathtools,graphicx}
\usepackage{epsfig,epstopdf,setspace,stackrel}
\usepackage[allcolors = blue,colorlinks]{hyperref}
\usepackage[abs]{overpic}
\usepackage{tikz-cd}

\newtheorem{theorem}{Theorem}[section]
\newtheorem{lemma}[theorem]{Lemma}

\newtheorem{example}[theorem]{Example}
\newtheorem{prop}[theorem]{Proposition}
\newtheorem{rem}[theorem]{Remark}
\newtheorem{corollary}[theorem]{Corollary}

\newtheorem{problem}[theorem]{Problem}

\newtheorem*{theorem*}{Main Theorem}

\def\rdiv{/}
\def\aut#1{\mathrm{Aut}(#1)}
\def\rmlt#1{\mathrm{Mlt}_r(#1)}
\def\col#1{\mathrm{Col}(#1)}
\def\inv{^{-1}}
\def\rdis#1{\mathrm{Dis}_r(#1)}

\title{Algebraic invariants of multi-virtual links}

\address[Kauffman]{Department of Mathematics, University of Illinois at Chicago, USA}
\email{{\rm loukau@gmail.com}}
\author{Louis H. Kauffman}

\address[Mukherjee]{Department of Mathematics, University of Denver, Colorado, USA}
\email{{\rm sujoymukherjee.math@gmail.com $|$ sujoy.mukherjee@du.edu}}
\author{Sujoy Mukherjee}

\address[Vojt\v{e}chovsk\'y]{Department of Mathematics, University of Denver, Colorado, USA}
\email{{\rm petr.vojtechovsky@du.edu}}
\author{Petr Vojt\v{e}chovsk\'y}

\date{\today}

\begin{document}

\begin{abstract}
Multi-virtual knot theory was introduced in $2024$ by the first author. In this paper, we initiate the study of algebraic invariants of multi-virtual links. After determining a generating set of (oriented) multi-virtual Reidemeister moves, we discuss the equivalence of multi-virtual link diagrams, particularly those that have the same virtual projections. We introduce operator quandles (that is, quandles with a list of pairwise commuting automorphisms) and construct an infinite family of connected operator quandles in which at least one third of right translations are distinct and pairwise commute. Using our set of generating moves, we establish the operator quandle coloring invariant and the operator quandle $2$-cocycle invariant for multi-virtual links, generalizing the well-known invariants for classical links. With these invariants at hand, we then classify certain small multi-virtual knots based on the existing tables of small virtual knots due to Bar-Natan and Green. Finally, to emphasize a key difference between virtual and multi-virtual knots, we construct an infinite family of pairwise nonequivalent multi-virtual knots, each with a single classical crossing. Many open problems are presented throughout the paper.
\end{abstract}

\maketitle

\section{Introduction}

Multi-virtual knot theory generalizes virtual knot theory \cite{Kau1} by employing multiple types of virtual crossings. From a topological point of view, virtual knot theory can be interpreted in terms of knots and links embedded in thickened surfaces up to knot theoretic equivalence in these surfaces and up to 1-handle stabilization. Multi-virtual knot theory can be similarly interpreted by adding labeled handles for each virtual crossing, subject to the restriction that only handles with the same label can merge, and with handle stabilization arising in relation to these operations.

Multi-virtual knot theory was initiated recently in \cite{Kau2}. In this paper we begin the investigation of invariants for multi-virtual links. Prominent among these are colorings and cocycle invariants based on operator quandles, which are quandles enhanced by a list of pairwise commuting automorphisms, one for each type of a virtual crossing.

The paper is organized as follows. In Section \ref{Sc:MVKT} we briefly recall the equivalence of multi-virtual links from \cite{Kau2}. Building upon \cite{Pol}, we derive a small generating set of Reidemeister moves for multi-virtual links and for oriented multi-virtual links. We prove that in the process of determining the equivalence of two multi-virtual link diagrams it is not necessary to introduce new types of virtual crossings not present in the two diagrams. We also point out that, unlike in the theory of classical and virtual knots, the total crossing number of a multi-virtual knot $K$ (that is, the minimal number of all classical and virtual crossings in a diagram of $K$) is not bounded above in terms of the classical crossing number of $K$. Consequently, the number of classical crossings is not a suitable organizing principle for a catalog of multi-virtual knots up to equivalence.

In Section \ref{Sc:Bell} we briefly discuss the problem of equivalence for all multi-virtual links with the same virtual projection (a diagram obtained by merging all types of virtual crossings into one type) and, more generally, the problem of equivalence of all multi-virtual links with given virtual projections $D_1$ and $D_2$. The number of cases that need to be considered is subject to combinatorial explosion and is equal to the Bell number $B_{n_1+n_2}$, where $n_i$ is the number of virtual crossings in $D_i$.

In Section \ref{Sc:Invariants} we introduce two invariants for multi-virtual links based on quandles and pairwise commuting quandle automorphisms. Given a multi-virtual link $K$ with virtual crossings of $n$ types, a quandle $Q$ and $n$ not necessarily distinct pairwise commuting automorphisms of $Q$, the first invariant counts the number of colorings of (any diagram of) $K$ based on these parameters. Unlike in classical knot theory, there might not exist any colorings for some choice of parameters, not even trivial colorings. The second invariant is a generalization of the well-known $2$-cocycle invariant \cite{CJKLS} which is based on quandle colorings and quandle $2$-cocycles.

To apply quandle colorings to multi-virtual links with $n$ types of virtual crossings, it is desirable to use quandles with at least $n$ distinct and pairwise commuting automorphisms, so that virtual crossings of distinct types are not conflated. We are therefore interested in quandles with many commuting automorphisms. We are particularly interested in connected quandles with many commuting and pairwise distinct right translations (which are automatically automorphisms). Some initial observations about such quandles are made in Section \ref{Sc:Quandles}. Among other results, we construct an infinite family of connected quandles $Q$ possessing $|Q|/3$ distinct and pairwise commuting right translations. We expect that this topic will be of independent interest for researchers working within algebraic theory of quandles.

Taking advantage of the existing tables of virtual knots \cite{Gre}, in Section \ref{Sc:Distinguish} we study a family of multi-virtual knots obtained from small virtual knots upon assigning in all possible ways the types to all virtual crossings. We distinguish the resulting multi-virtual knots by means of the invariants from Section \ref{Sc:Invariants}. Note that we do not attempt to comprehensively classify small multi-virtual knots here, but we offer some remarks to that effect.

This paper should be seen as a preliminary exploration of algebraic invariants of multi-virtual links and the classification of multi-virtual knots and links. Open problems abound. We collect some of them throughout the paper.

\subsection*{Acknowledgments}

This project started while the first author visited the University of Denver to deliver the 2024 Howe Lecture in Mathematics. The third author is supported by the Simons Foundation Mathematics and Physical Sciences Collaboration Grant for Mathematicians no.~855097.

\section{Multi-virtual knot theory}\label{Sc:MVKT}

Let $T$ be a fixed set whose elements will be called \emph{types} and denoted generically by $\alpha$, $\beta$, $\gamma$, etc. A \emph{multi-virtual link diagram} (of type $T$) is a planar drawing of a link in which classical crossings are depicted by overpasses and underpasses as usual, and where every virtual crossing is assigned a type from $T$. The type of a virtual crossing will be displayed either by placing a label next to the generic virtual crossing marker (a circle), or by using visually different markers for virtual crossings of different types; see Figure \ref{Fg:MVKD}. We do not demand that all elements of $T$ must occur as types in a diagram. 

\begin{figure}[!ht]
\centering
\begin{overpic}[width = 0.6\textwidth]{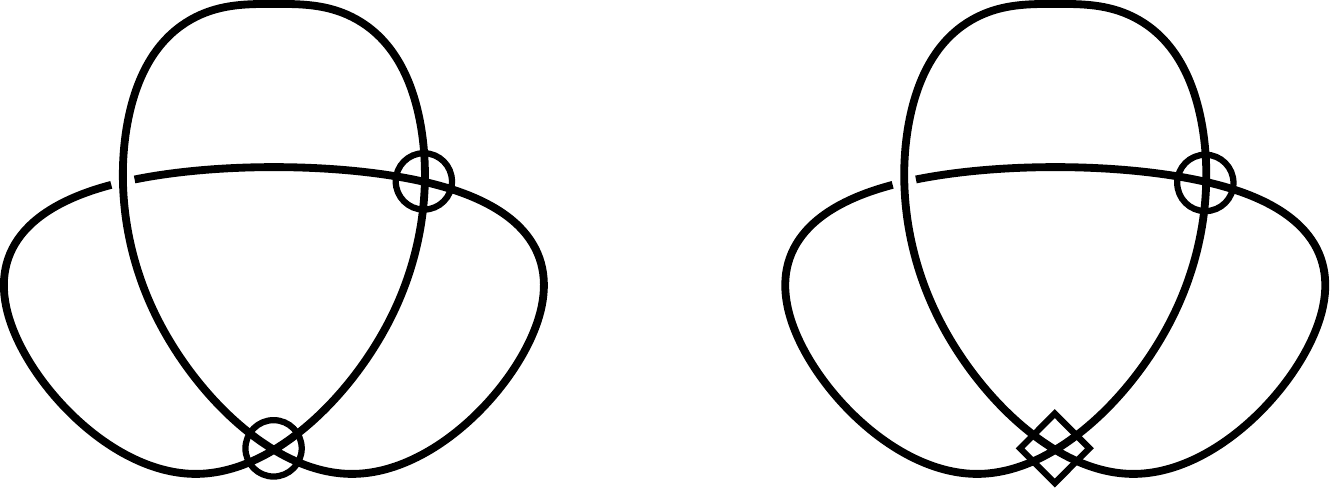}
\put(54,18){$\alpha$}
\put(77,52){$\beta$}
\end{overpic}
\caption{Two conventions for drawing multi-virtual link diagrams.}\label{Fg:MVKD}
\end{figure}

Two multi-virtual link diagrams are said to be \emph{equivalent} if one is obtained from the other by finitely many applications of isotopy, classical Reidemeister moves and multi-virtual detour moves. A \emph{multi-virtual detour move} is a move by which a consecutive sequence (possibly empty) of virtual crossings of the same type is excised from the diagram and the endpoints of this excision are reconnected by a diagrammatic arc that intersects the rest of the diagram in a sequence of virtual crossings of this same type. More precisely, a multi-virtual detour move allows a strand $s$ meeting $m$ strands at virtual crossings of the same type to be moved across an $m,n$-tangle, resulting in $n$ new virtual crossings of this same type, while freely adding/removing loops of this same type to the strand. A multi-virtual detour move with $m = 4$ and $n = 2$ is illustrated in Figure \ref{Fg:DetourMove}, where the gray box contains an $m,n$-tangle. 

\begin{figure}[!ht]
\includegraphics[width = 0.6\textwidth]{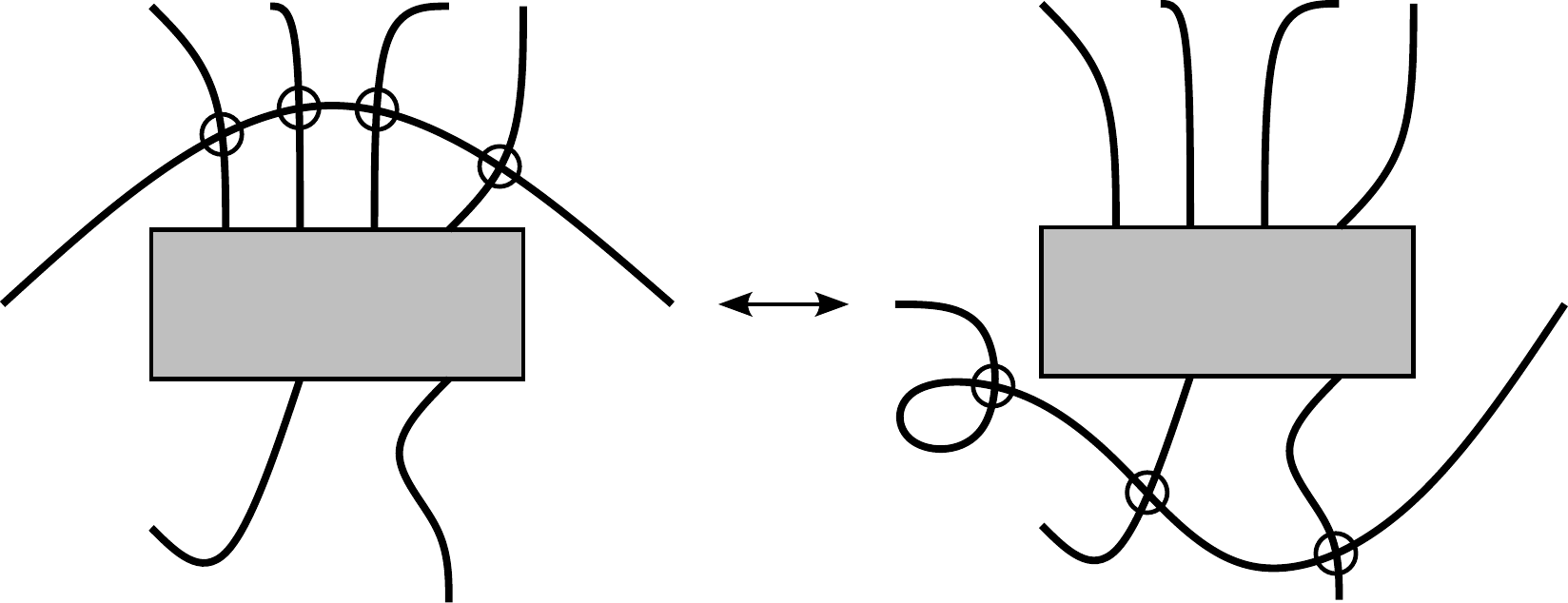}
\caption{A multi-virtual detour move.}\label{Fg:DetourMove}
\end{figure}

It is difficult to work with the multi-virtual detour move since it represents an infinite family of moves. Hence, it is desirable to have access to a small generating set of (oriented) Reidemeister moves, for instance while establishing that certain properties of diagrams are in fact invariants of links.

In this section we obtain a generating set of Reidemeister moves that is equivalent to the multi-virtual detour move (see Figure \ref{Fg:NonclassicalRmoves} and Proposition \ref{Pr:DetourEquiv}), a generating set of Reidemeister moves for multi-virtual links (see Theorem \ref{Th:Basis}), and a generating set of Reidemeister moves for oriented multi-virtual links (see Figure \ref{Fg:Minimalmoves} and Theorem \ref{Th:OrientedBasis}).

Since the multi-virtual Reidemeister 1 move can introduce a new type of a virtual crossing to a diagram, it is not immediately clear whether new types of virtual crossings need to be considered for the equivalence problem. We prove that no new types are needed (see Proposition \ref{Pr:Projection}).

We conclude this section by studying the equivalence problem of multi-virtual link diagrams with given virtual projections, by constructing an infinite family of interesting pairwise nonequivalent multi-virtual links with no classical crossings (distinguished by means of the chromatic bracket polynomial). An infinite family of pairwise nonequivalent multi-virtual knots, each with a single classical crossing, is given later in Section \ref{Sc:Distinguish}.

\subsection{A generating set of Reidemeister moves for multi-virtual links}

\begin{figure}[!ht]
\centering
\begin{overpic}[width = 0.7\textwidth]{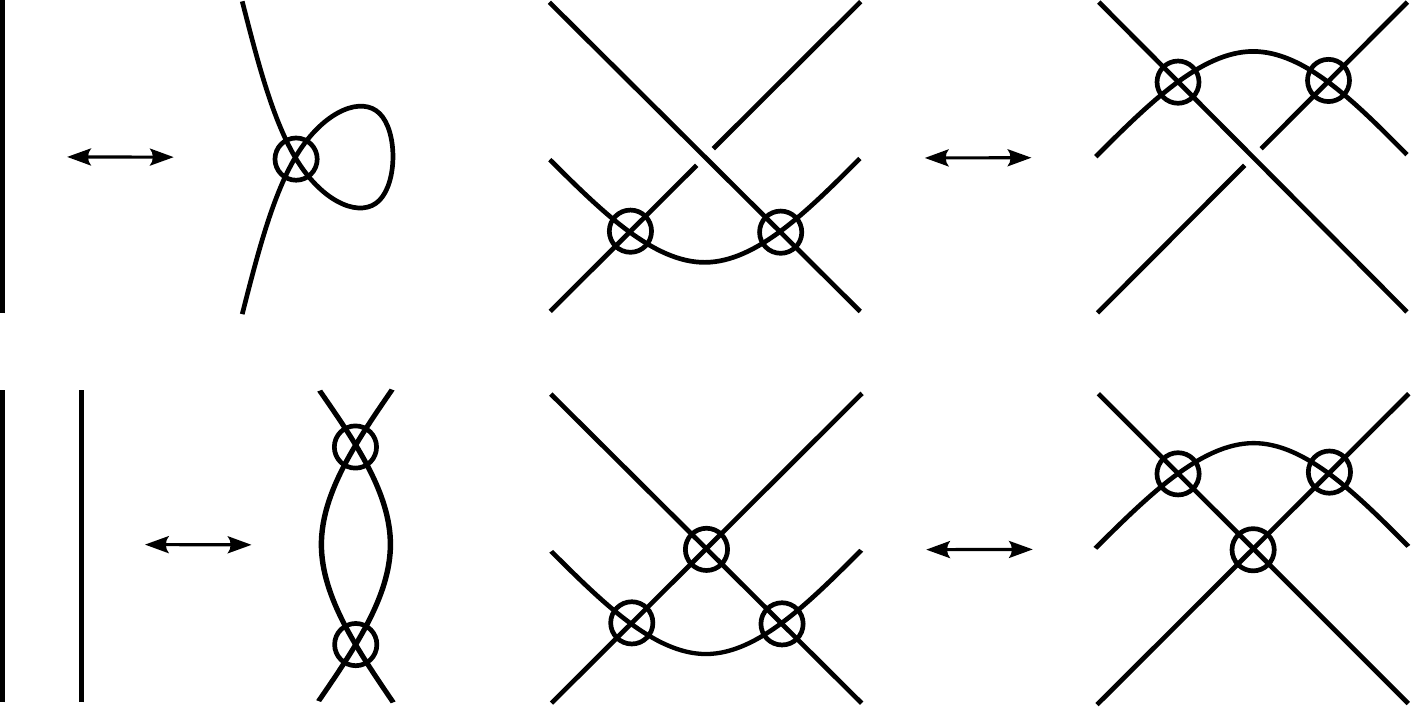}

\put(23,131){$v1$}
\put(41,40){$v2$}
\put(220,131){$cv3$}
\put(218,40){$mv3$}

\put(172,34){$\beta$}
\put(299,34){$\beta$}

\put(142,96){$\alpha$}
\put(179,96){$\alpha$}
\put(142,5){$\alpha$}
\put(179,5){$\alpha$}

\put(272,155){$\alpha$}
\put(307,155){$\alpha$}
\put(272,64){$\alpha$}
\put(307,64){$\alpha$}

\put(54,125){$\alpha$}
\put(67,58){$\alpha$}
\put(67,12){$\alpha$}

\end{overpic}
\caption{The multi-virtual detour move is equivalent to the four moves $v1$, $v2$, $cv3$ and $mv3$.}\label{Fg:NonclassicalRmoves}
\end{figure}
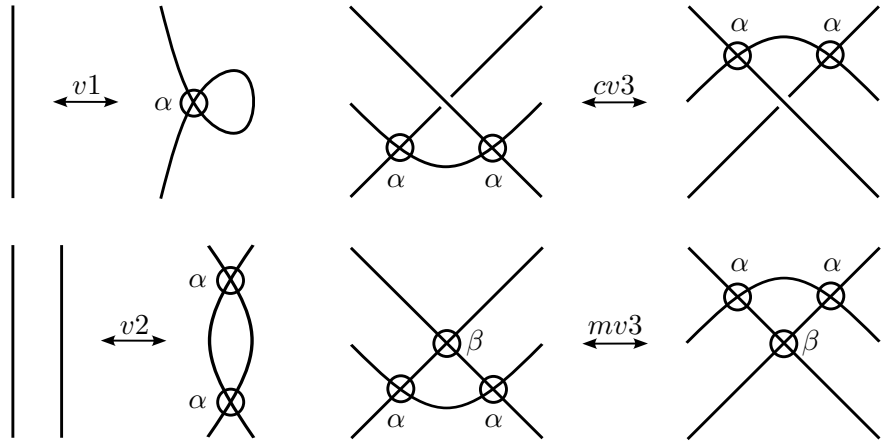

Each of the moves in Figure \ref{Fg:NonclassicalRmoves} is a special case of the multi-virtual detour move. Indeed, since the multi-virtual detour move allows the addition/removal of loops by definition, the $v1$ move can be obtained as a special case. Figure \ref{Fg:R2R3asdetour} illustrates the moves $v2$ and $mv3$ as multi-virtual detour moves. The $cv3$ move is obtained similarly by replacing the virtual crossing of type $\alpha$ in the tangle for the $mv3$ move with a classical crossing. 

\begin{figure}[!ht]
\centering
\begin{overpic}[width = 0.6\textwidth]{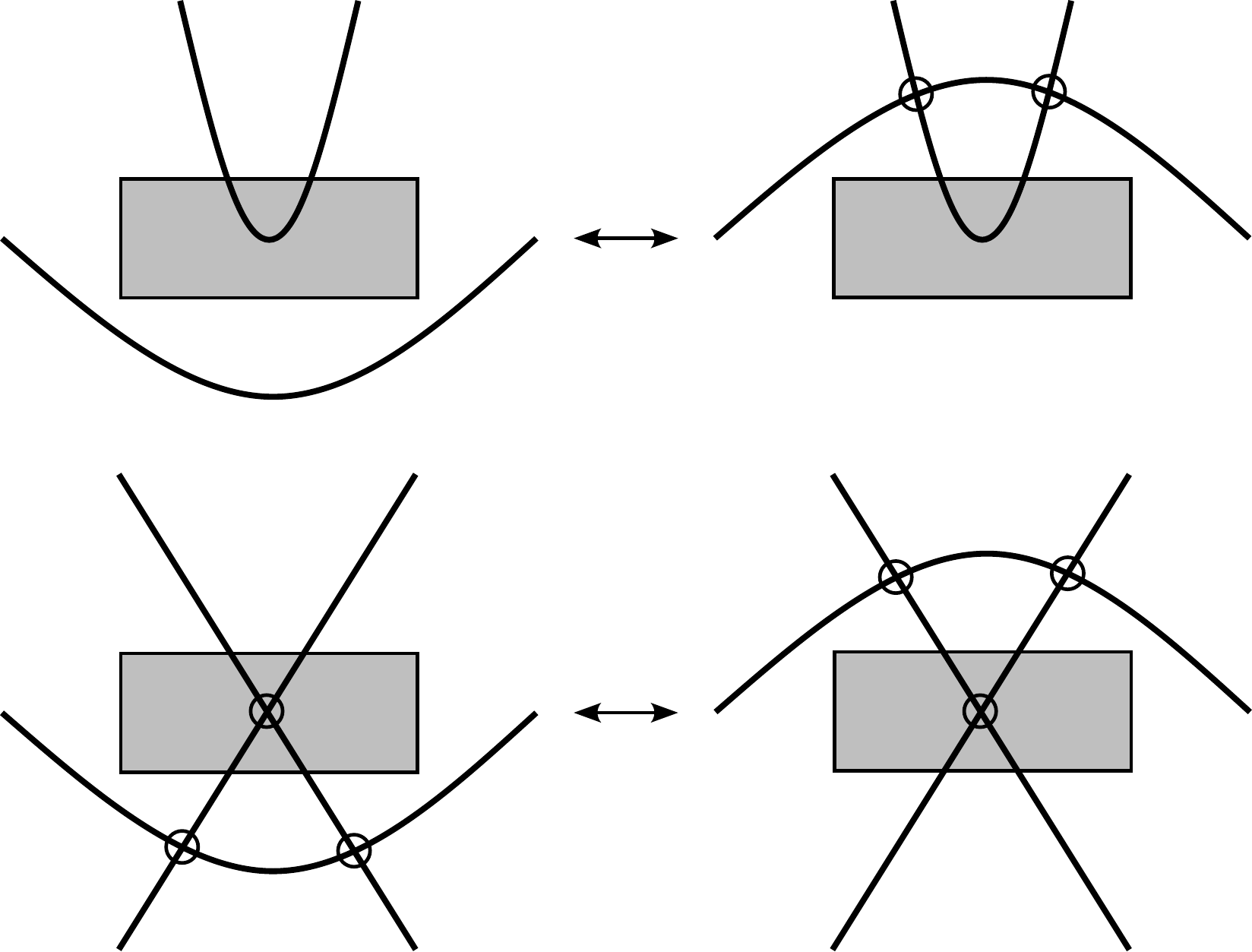}

\put(135,164){$v2$}
\put(131,58){$mv3$}

\put(67,52){$\alpha$}
\put(228,52){$\alpha$}

\put(194,194){$\beta$}
\put(243,194){$\beta$}

\put(26,16){$\alpha$}
\put(87,16){$\alpha$}
\put(188,84){$\alpha$}
\put(247,84){$\alpha$}

\end{overpic}
\caption{The $v2$ and $mv3$ moves realized as multi-virtual detour moves.}\label{Fg:R2R3asdetour}
\end{figure}
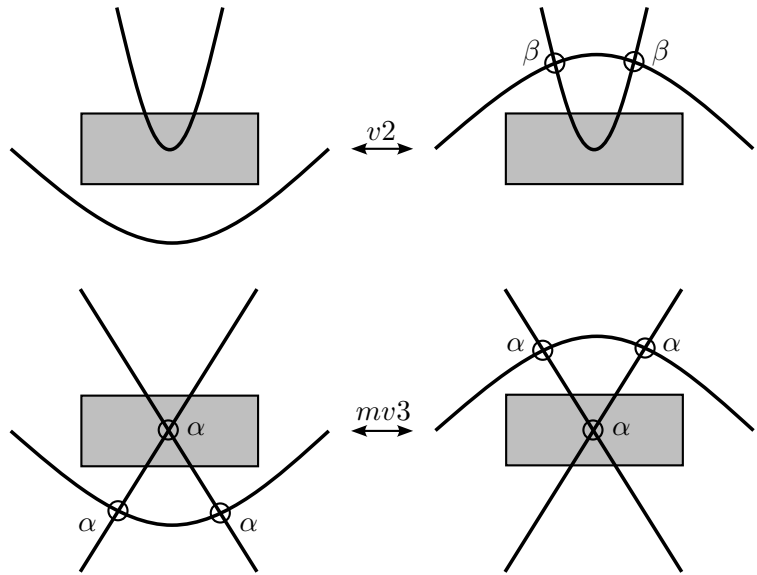

Proving the converse implication is also not difficult. Suppose that in a multi-virtual detour move the strand $s$ meets the $m$ strands of the $m,n$-tangle at virtual crossings of type $\alpha$. Observe that, using isotopy, any $m,n$-tangle can be broken into horizontal strips such that when scanning from top to bottom, we encounter one cap or one cup or one crossing (classical or virtual) in every strip, as illustrated in Figure \ref{Fg:DetourasRmoves}. Then, to move the strand $s$ across a strip, we use the $v2$ move when encountering caps or cups (noting that any virtual crossings added due to a $v2$ move will be of type $\alpha$), the $cv3$ move when encountering a classical crossing, and the $mv3$ move when encountering a virtual crossing. For instance, in the example depicted in Figure \ref{Fg:DetourasRmoves}, the multi-virtual detour move across the $3,5$-tangle shown in the figure consists of the following sequence of moves (from top to bottom): $v2$, $mv3$, $cv3$, $v2$, $cv3$, $v2$, $cv3$, $cv3$, $mv3$.

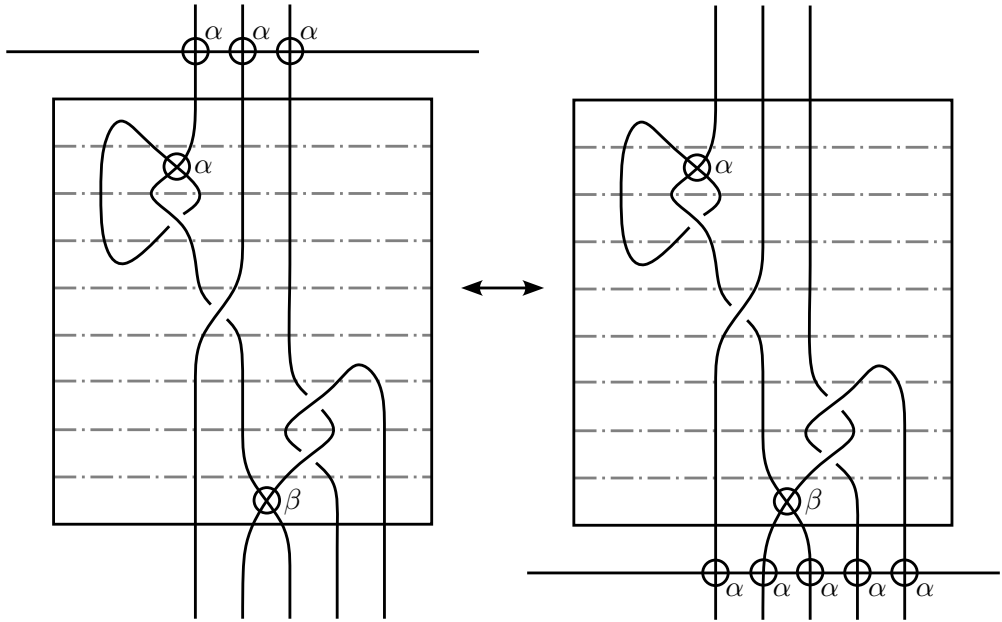
\begin{figure}[!ht]
\centering
\begin{overpic}[width = 0.8\textwidth]{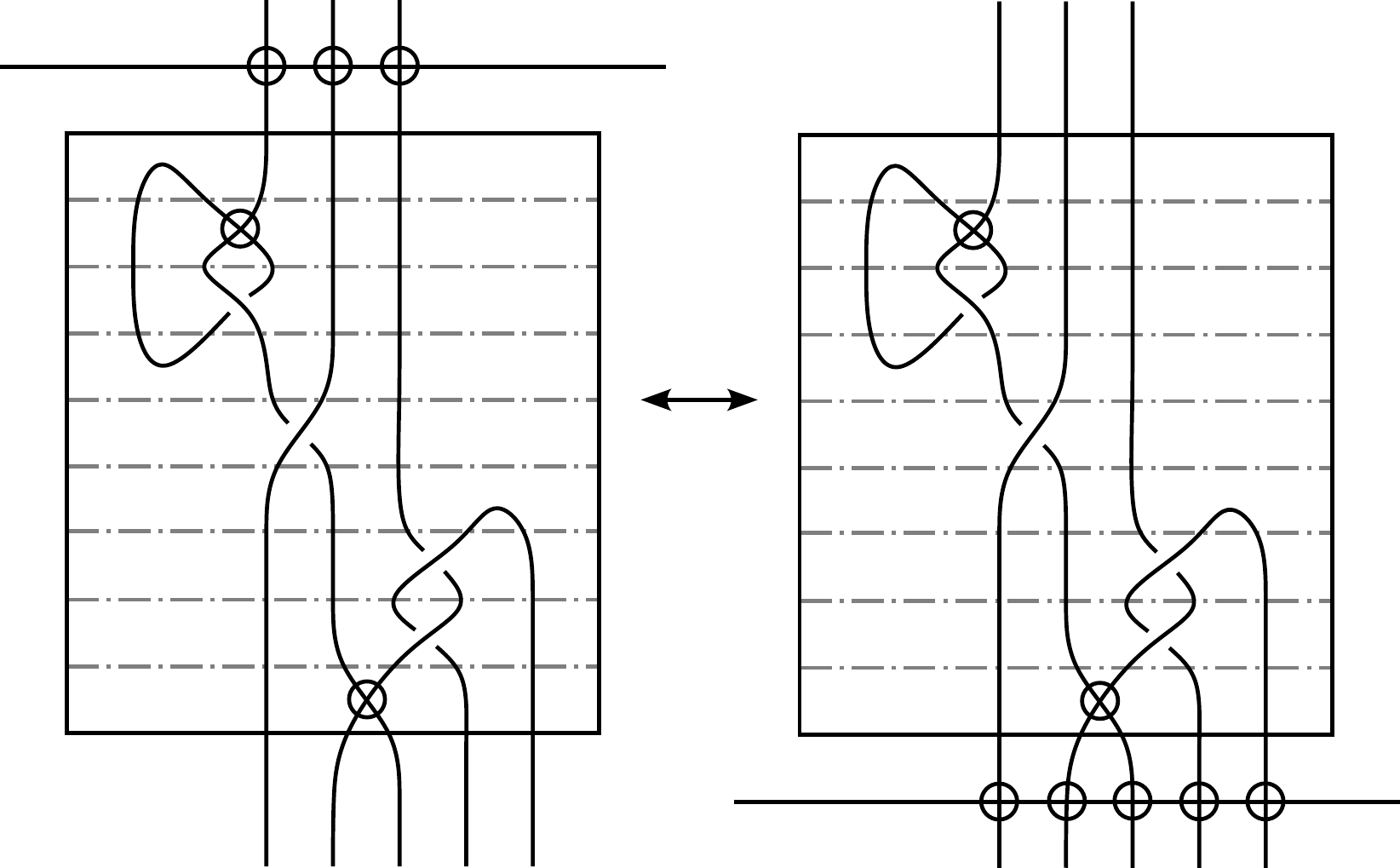}

\put(75,220){$\alpha$}
\put(93,220){$\alpha$}
\put(111,220){$\alpha$}

\put(71,169){$\alpha$}
\put(268,169){$\alpha$}

\put(105,42){$\beta$}
\put(302,42){$\beta$}

\put(272,9){$\alpha$}
\put(290,9){$\alpha$}
\put(308,9){$\alpha$}
\put(326,9){$\alpha$}
\put(344,9){$\alpha$}

\end{overpic}
\caption{Example of a multi-virtual detour move realized as a sequence of the moves $v1$, $v2$, $v3$, $cv3$, and $mv3$.}\label{Fg:DetourasRmoves}
\end{figure}

We have proved the following results:

\begin{prop}\label{Pr:DetourEquiv}
The multi-virtual detour move is equivalent to the moves $v1$, $v2$, $cv3$ and $mv3$ depicted in Figure \ref{Fg:NonclassicalRmoves}.
\end{prop}

\begin{theorem}\label{Th:Basis}
Two multi-virtual link diagrams are equivalent if and only if one is obtained from the other by finitely many applications of isotopy, classical Reidemeister moves, and the multi-virtual Reidemeister moves $v1$, $v2$, $cv3$ and $mv3$. 
\end{theorem}

We will refer to the Reidemeister moves of Theorem \ref{Th:Basis} as \emph{multi-virtual Reidemeister moves}.

\subsection{A generating set of Reidemeister moves for oriented multi-virtual links}

Polyak showed in \cite{Pol} that any two diagrams of an oriented classical link are related by isotopy and a finite sequence of the moves $\Omega1a,$ $\Omega1b,$ $\Omega2a,$ and $\Omega3a$, using his notation. See Figure \ref{Fg:Minimalmoves} for a depiction of the $\Omega$-moves.

Polyak's proof can be extended to virtual links. Note that there is no Reidemeister 1 move of mixed type and that Reidemeister 2 moves of mixed type are forbidden in multi-virtual knot theory \cite{Kau2}. Therefore, to obtain a generating set of oriented Reidemeister moves for multi-virtual knot theory, we can start by listing several different Reidemeister 3 moves of multi-virtual type in Figure \ref{Fg:VirtualR3moves}, where we have adapted Polyak's nomenclature from classical moves to multi-virtual moves by changing $\Omega$ to $mv$. For example, the analogue of the $\Omega3e$ move in Polyak's list is the $mv3e$ move in our list. 

\begin{figure}[!ht]
\centering
\begin{overpic}[width = \textwidth]{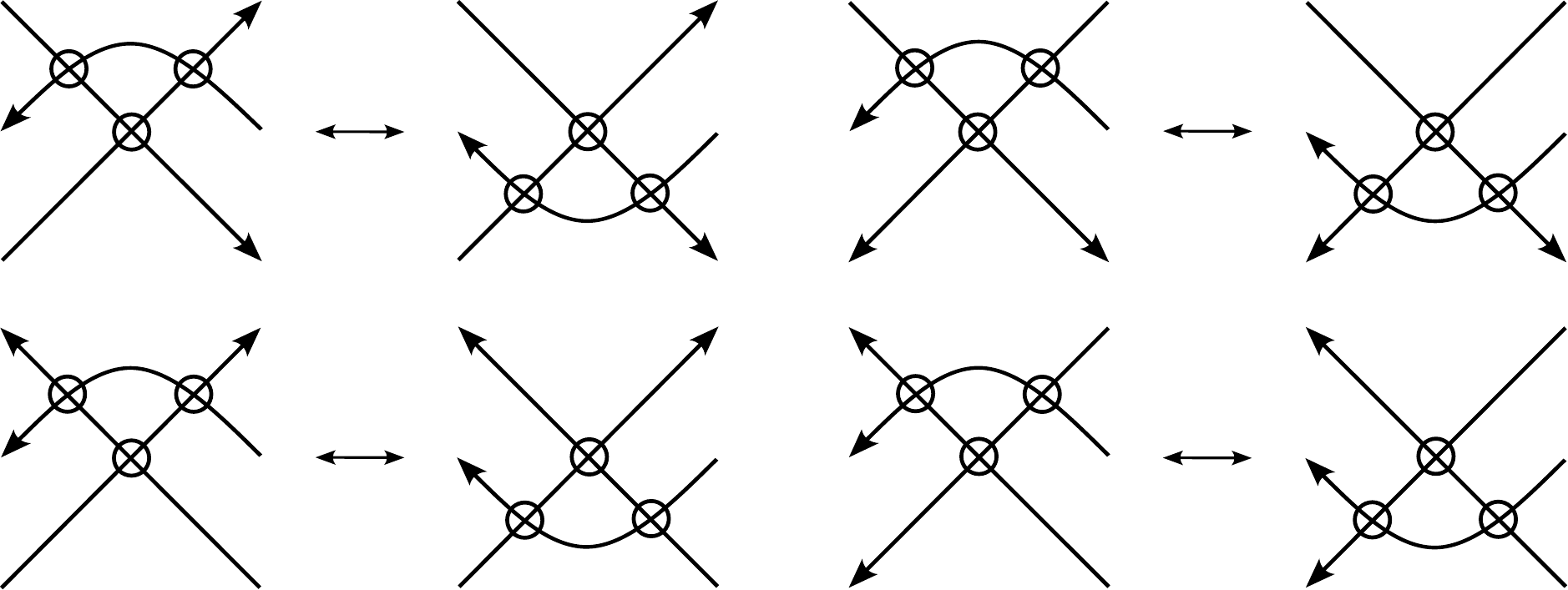}
\put(95,142){$mv3a$}
\put(349,142){$mv3b$}
\put(95,45){$mv3c$}
\put(349,45){$mv3e$}

\put(36,120){$\beta$}
\put(173,148){$\beta$}
\put(290,120){$\beta$}
\put(427,148){$\beta$}
\put(36,23){$\beta$}
\put(173,50){$\beta$}
\put(290,23){$\beta$}
\put(427,50){$\beta$}

\put(5,154){$\alpha$}
\put(67,154){$\alpha$}
\put(259,154){$\alpha$}
\put(321,154){$\alpha$}

\put(142,115){$\alpha$}
\put(204,115){$\alpha$}
\put(396,115){$\alpha$}
\put(458,115){$\alpha$}

\put(5,57){$\alpha$}
\put(67,57){$\alpha$}
\put(259,57){$\alpha$}
\put(321,57){$\alpha$}

\put(142,18){$\alpha$}
\put(204,18){$\alpha$}
\put(396,18){$\alpha$}
\put(458,18){$\alpha$}
\end{overpic}
\caption{Multi-virtual Reidemeister 3 moves.}\label{Fg:VirtualR3moves}
\end{figure}
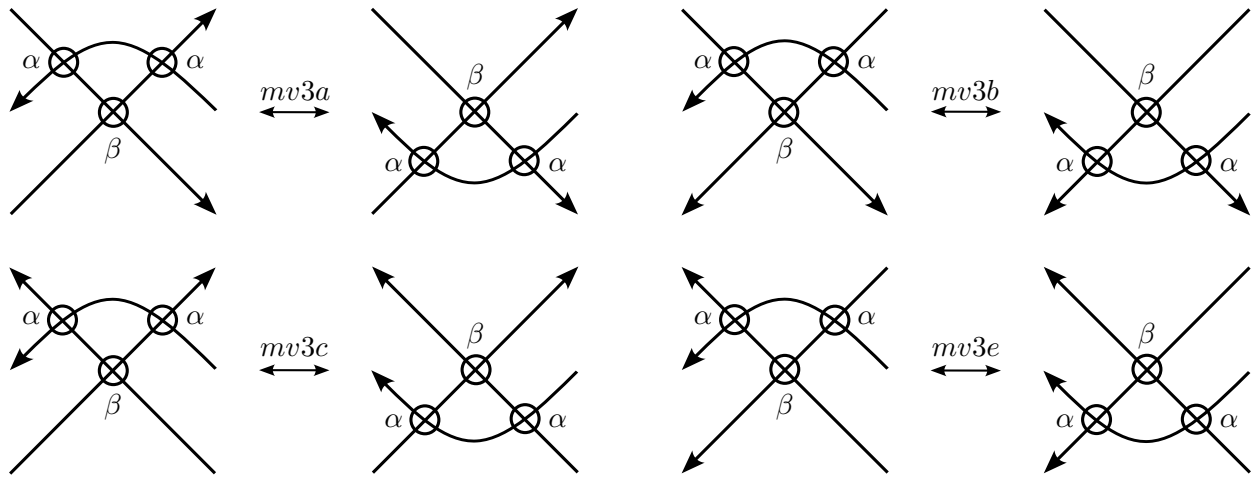

\begin{lemma}\label{Lm:Aux}
The moves $mv3b$, $mv3c$ and $mv3e$ follow from the virtual Reidemeister 2 moves and the $mv3a$ move.
\end{lemma}

\begin{proof}
The proof, an adaptation from the classical case, follows from Figure \ref{Fg:Mv3b3c3e}. In the figure, all non-labeled virtual crossings are of type $\alpha$.

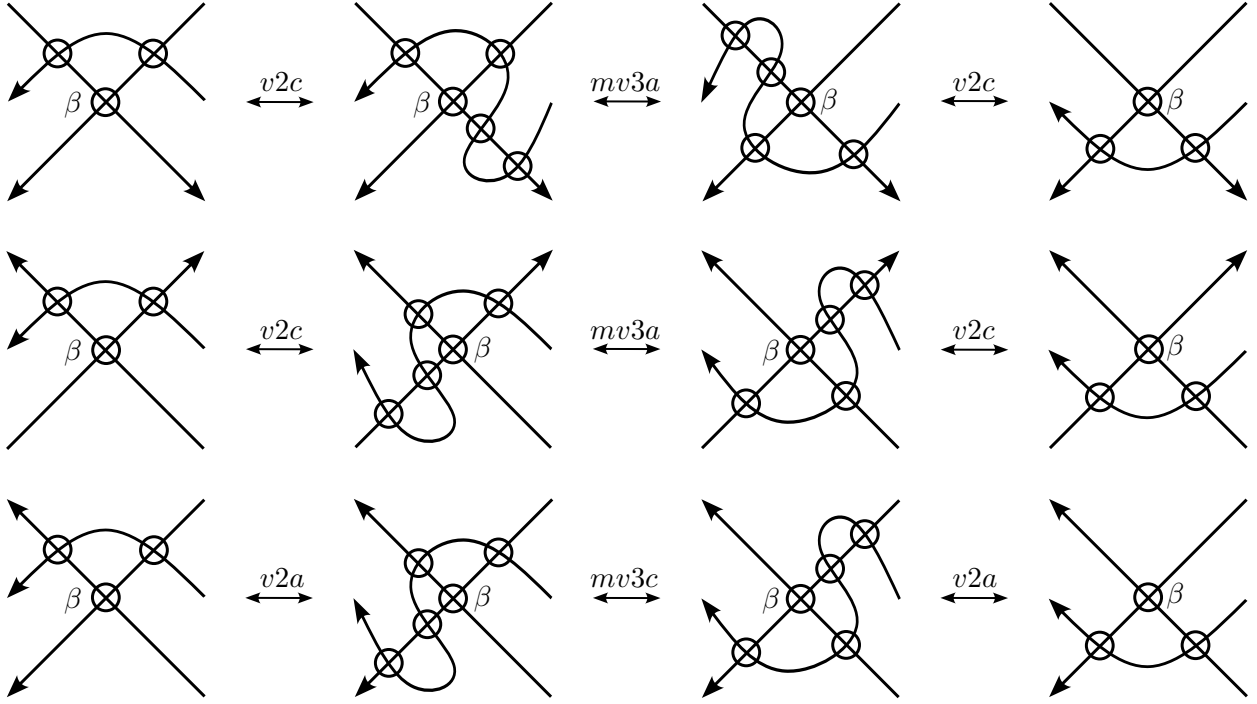
\begin{figure}[!ht]
\centering
\begin{overpic}[width = \textwidth]{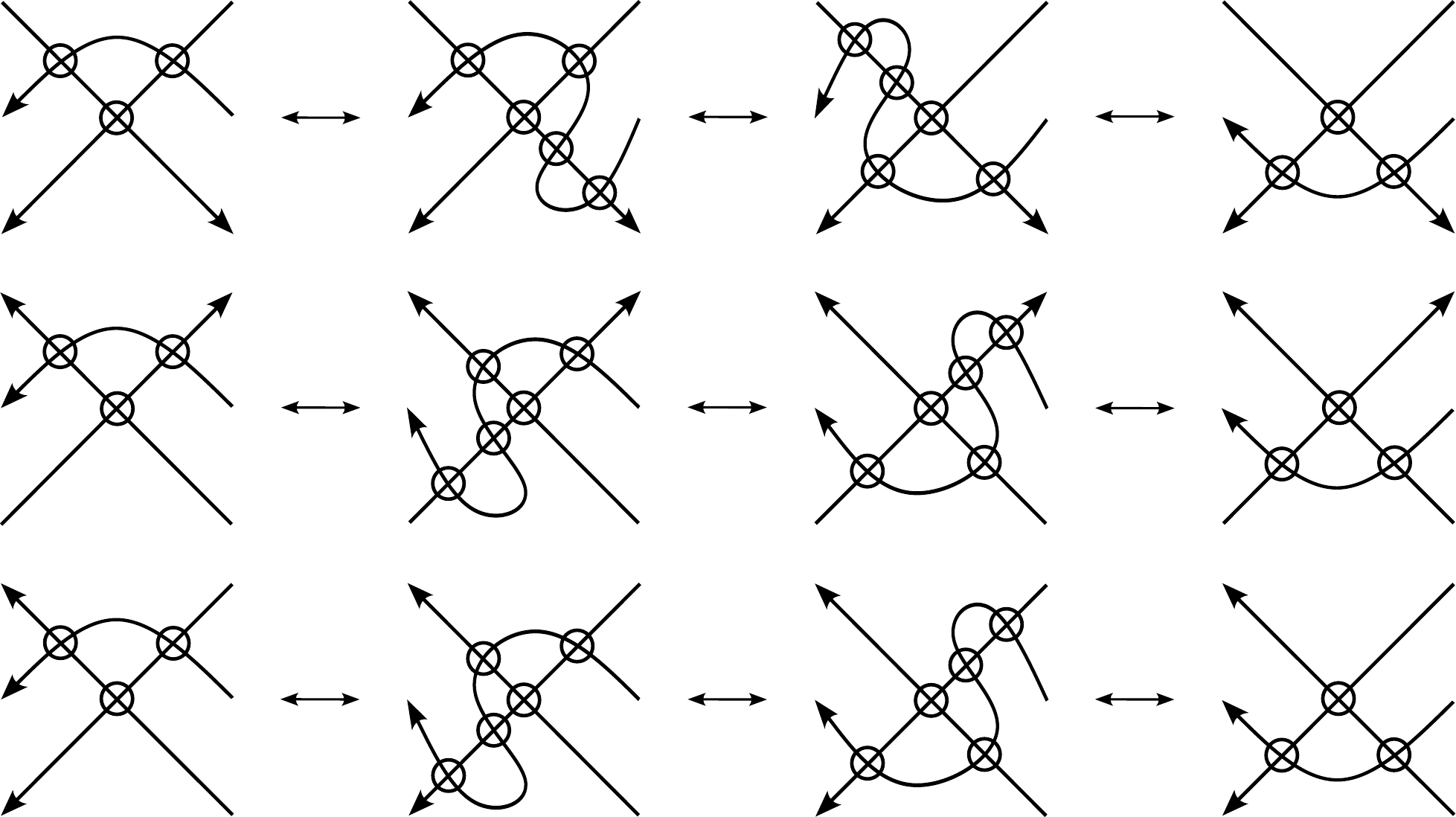}

\put(96,230){$v2c$}
\put(96,136){$v2c$}
\put(96,42){$v2a$}
\put(221,230){$mv3a$}
\put(221,136){$mv3a$}
\put(221,42){$mv3c$}
\put(358,230){$v2c$}
\put(358,136){$v2c$}
\put(358,42){$v2a$}

\put(22,222){$\beta$}
\put(22,128){$\beta$}
\put(22,34){$\beta$}
\put(154,222){$\beta$}
\put(177,129){$\beta$}
\put(177,35){$\beta$}
\put(308,223){$\beta$}
\put(286,128){$\beta$}
\put(286,34){$\beta$}
\put(439,223){$\beta$}
\put(439,130){$\beta$}
\put(439,36){$\beta$}

\end{overpic}
\caption{Proof of Lemma \ref{Lm:Aux}. The moves $mv3b$, $mv3c$, and $mv3e$ (from top to bottom) follow from the virtual Reidemeister 2 moves and the $mv3a$ move.}\label{Fg:Mv3b3c3e}
\end{figure}

\end{proof}

Therefore, we have the following theorem:

\begin{theorem}\label{Th:OrientedBasis}
Two oriented multi-virtual link diagrams are equivalent if and only if one is obtained from the other by finitely many applications of isotopy and the moves $\Omega1a$, $\Omega1b$, $\Omega2a$, $\Omega3a$, $v1a$, $v1b$, $v2a$, $cv3a$ and $mv3a$, depicted in Figure \ref{Fg:Minimalmoves}.
\end{theorem}

We will refer to the oriented Reidemeister moves of Theorem \ref{Th:OrientedBasis} as \emph{oriented multi-virtual Reidemeister moves}, or just as \emph{multi-virtual Reidemeister moves} when it is clear from the context that we are working with oriented multi-virtual links.

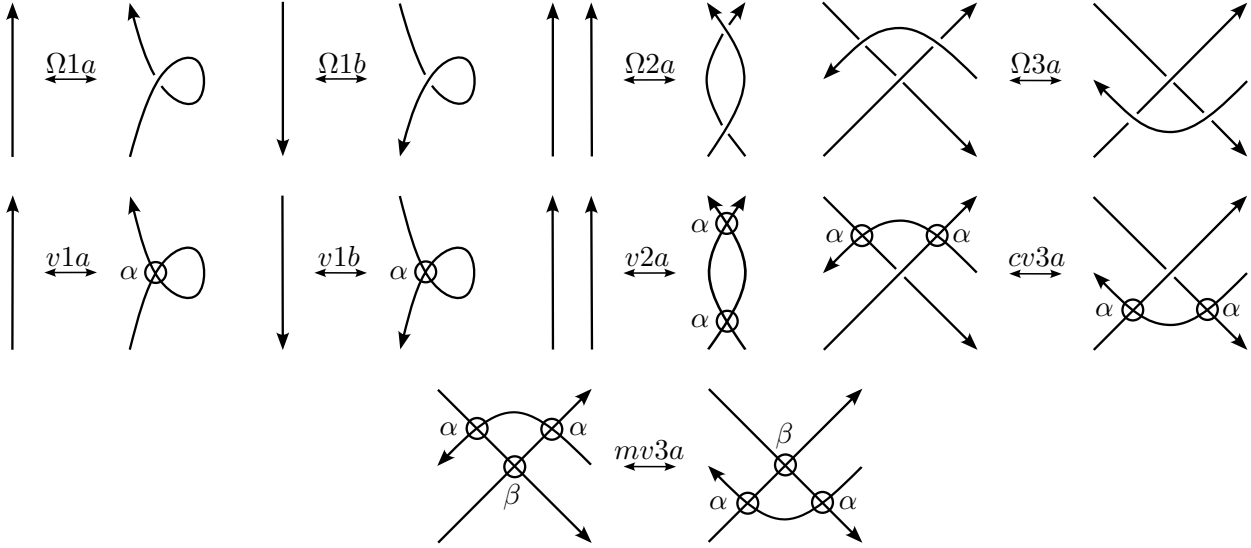
\begin{figure}[!ht]
\centering
\begin{overpic}[width = \textwidth]{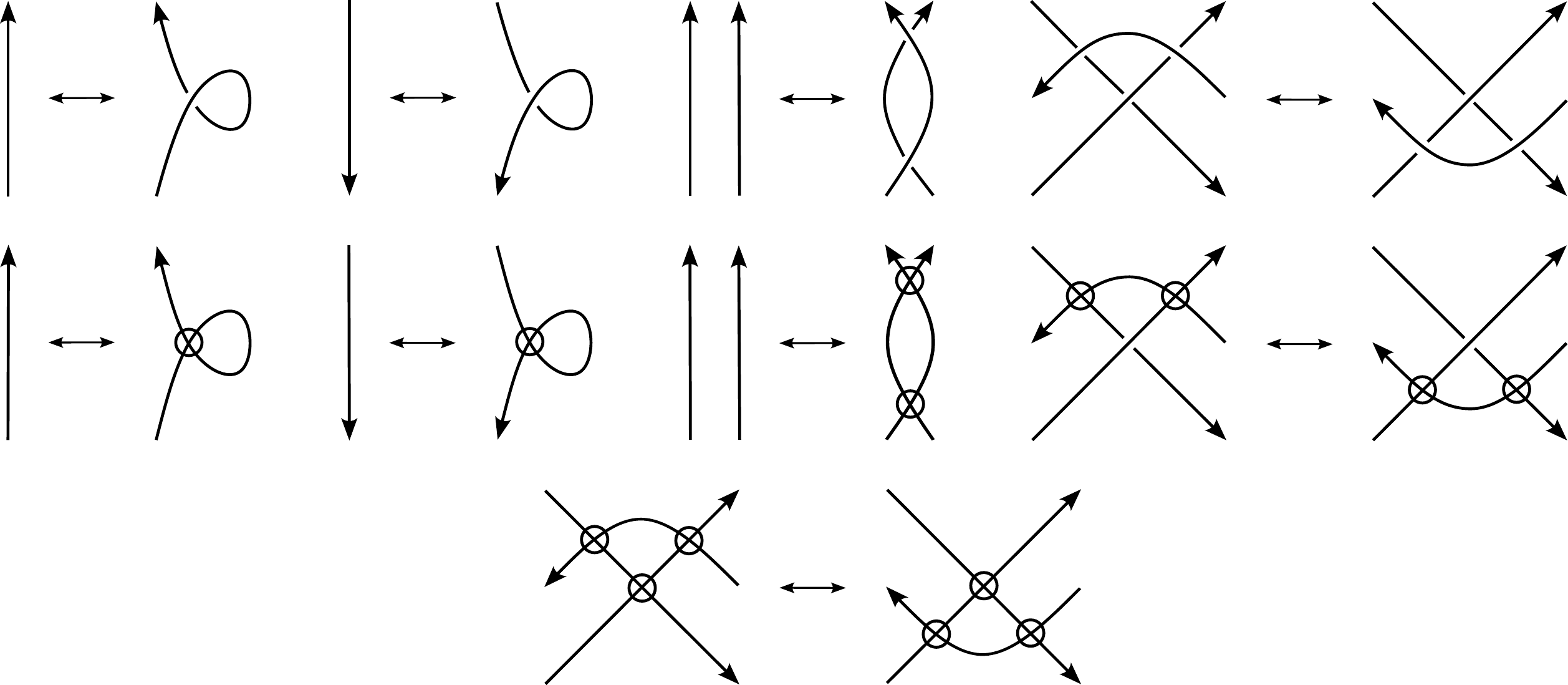}

\put(15,178){$\Omega1a$}
\put(118,178){$\Omega1b$}
\put(234,178){$\Omega2a$}
\put(380,178){$\Omega3a$}

\put(15,105){$v1a$}
\put(118,105){$v1b$}
\put(234,105){$v2a$}
\put(379,105){$cv3a$}

\put(230,32){$mv3a$}

\put(43,100){$\alpha$}

\put(145,100){$\alpha$}

\put(259,82){$\alpha$}
\put(259,118){$\alpha$}

\put(310,114){$\alpha$}
\put(358,114){$\alpha$}
\put(412,86){$\alpha$}
\put(460,86){$\alpha$}

\put(164,41){$\alpha$}
\put(213,41){$\alpha$}
\put(188,15){$\beta$}
\put(266,13){$\alpha$}
\put(315,13){$\alpha$}
\put(291,38){$\beta$}

\end{overpic}
\caption{A generating set of Reidemeister moves for oriented multi-virtual links.}\label{Fg:Minimalmoves}
\end{figure}

\subsection{New types of virtual crossings}

It can certainly happen that two multi-virtual link diagrams are equivalent even if they use virtual crossings of different types, for instance when they are both equivalent to the unknot, as illustrated in Figure \ref{Fg:TwoUnknots}. 

\begin{figure}[ht]
\centering
\begin{overpic}[width = 0.3\textwidth]{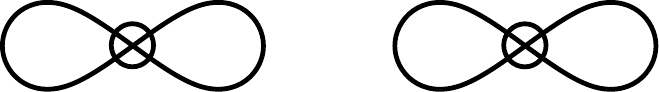}
\put(25,-5){$\alpha$}
\put(109,-5){$\beta$}
\end{overpic}
\caption{Two unknots.}\label{Fg:TwoUnknots}
\end{figure}

Note that new types of virtual crossings can be introduced to a diagram by the $v1$ move. The next result states that while determining the equivalence of two multi-virtual link diagrams, it is not necessary to introduce virtual crossings of types that are not yet present in at least one of the diagrams.

For a multi-virtual link diagram $D$, denote by $T(D)\subseteq T$ the set of all types that occur in $D$.

\begin{theorem}\label{Th:NoNewTypes}
Let $D_1$ and $D_2$ be equivalent (oriented) multi-virtual link diagrams. Then there exists a sequence $E_1,\dots,E_n$ of (oriented) multi-virtual link diagrams such that
\begin{enumerate}
    \item[(i)] $D_1=E_1$ and $D_2=E_n$,
    \item[(ii)] for every $1\le i<n$, $E_{i+1}$ is obtained from $E_i$ by a (oriented) multi-virtual Reidemeister move and isotopy, and
    \item[(iii)] for every $1\le i\le n$, $T(E_i)\subseteq T(D_1)\cup T(D_2)$.
\end{enumerate}
\end{theorem}
\begin{proof}
Suppose first that $T(D_1)\cup T(D_2)\ne\emptyset$. Since $D_1$ is equivalent to $D_2$, it is possible to transform $D_1$ into $D_2$ by a finite sequence of multi-virtual Reidemeister moves and isotopy. Suppose that in that sequence a new type $\alpha\not\in T(D_1)\cup T(D_2)$ is introduced. Let $\beta$ be any element of $T(D_1)\cup T(D_2)\ne\emptyset$. We claim that the sequence obtained from the original sequence by replacing every occurrence of $\alpha$ by $\beta$ is a valid sequence of multi-virtual Reidemeister moves for the equivalence argument in question. Indeed, the only way in which $\alpha$ essentially comes into play is in a multi-virtual detour move where several consecutive virtual crossings of type $\alpha$ are excised and replaced with an arc containing only virtual crossings of type $\alpha$. Such a move can certainly be performed with $\beta$ in place of $\alpha$.

Now suppose that $T(D_1)\cup T(D_2)=\emptyset$ so that $D_1$ and $D_2$ are classical link diagrams that are equivalent as multi-virtual link diagrams. Associated with any classical link diagram $K$ is the Wirtinger presented fundamental group $\pi(K)$ and a longitude elements $\lambda(K)\in\pi(K).$ The longitudes are expressed as products of meridianal generators of $\pi(K)$, obtained by walking along each component of the diagram $D$ and writing, in order of encounter, the names of the arcs that are passed under in the walk, with the name written with exponent $1$ if the overpass goes to the right and $-1$ if the overpass goes to the left. This method produces a longitude element of the fundamental group for each component of the link $K$. A result of Waldhausen \cite{Wal} states that the isotopy class of a classical link is determined by its fundamental group and the so-called \emph{peripheral subgroup}, which is the normal subgroup of the fundamental group generated by the longitudes and one choice of meridian generator for each link component. 

Note that the Wirtinger presentation for a fundamental group of a classical link is extended to multi-virtual diagrams by ignoring all virtual crossings. It is then clear that both the fundamental group and the peripheral subgroup are preserved under multi-virtual isotopy. Therefore, if $D_1$ and $D_2$ are equivalent as multi-virtual links then $D_1$ and $D_2$ have isomorphic fundamental groups and peripheral subgroups. By the result of Waldhausen, $D_1$ and $D_2$ are equivalent as classical links. Hence there exists the desired sequence $D_1=E_1,E_2,\dots,E_n=D_2$ such that $T(E_i)=\emptyset$ for every $1\le i\le n$.
\end{proof}

The following result is an immediate consequence of Theorem \ref{Th:NoNewTypes}.

\begin{corollary}\label{Cr:NoNewTypes}
Let $D_1$, $D_2$ be classical (resp. virtual) link diagrams that are equivalent as multi-virtual link diagrams. Then they are equivalent as classical (resp. virtual) link diagrams.
\end{corollary}

\subsection{The classical crossing number and the total crossing number}\label{Ss:Family}

For a multi-virtual link $L$, let the \emph{classical crossing number} $\mathrm{cr}_c(L)$ be the smallest number of classical crossings among all diagrams of $L$, and let the \emph{total crossing number} $\mathrm{cr}_t(L)$ be the smallest number of all (classical and virtual) crossings among all diagrams of $L$. If $\mathrm{cr}_c(L)=n$, we will say that $L$ \emph{has $n$ classical crossings}, and if $\mathrm{cr}_t(L)=n$, we will say that $L$ \emph{has $n$ crossings}.

The tables of classical knots are usually organized by the classical crossing number. The table \cite{Gre} of virtual knots is likewise organized by the classical crossing number. Implicit in the cataloging method of \cite{Gre} is the fact that, up to equivalence, there are only finitely many virtual knots with $n$ classical crossings---see Corollary \ref{Cr:VirtualGraph}.

This is not true for multi-virtual knots, even if $|T|=2$, as witnessed by the infinite family of pairwise nonequivalent multi-virtual knots with a single classical crossing constructed in Subsection \ref{Ss:SingleClassical}. The construction appears late in the paper since we will first need to develop the operator quandle invariant in order to distinguish the knots.

Nevertheless, it remains true that, up to equivalence, there are only finitely many multi-virtual knots with $n$ crossings, \emph{provided} that the set $T$ of types is finite---see Proposition \ref{Pr:MVBound}. We shall see in Subsection \ref{Ss:SameVirtual} that the equivalence problem for two multi-virtual links with $n$ virtual crossings each can be fully resolved by considering at most $2n$ types at a time. It is therefore safe to catalog multi-virtual knots by their total crossing number.

\begin{prop}\label{Pr:MVBound}
Let $T$ be a finite set of types. Then, up to equivalence, there are only finitely many multi-virtual knots of type $T$ with $n$ crossings.
\end{prop}
\begin{proof}
Let $D$ be a diagram of a multi-virtual knot $K$. Suppose that there are $n$ crossings in $D$ (so that $\mathrm{cr}_t(K)\le n$). The shadow $G$ of $D$ is a connected $4$-regular planar graph with $n$ vertices. Up to planar isotopy, the number of connected graphs with $n$ vertices is finite. Given a connected $4$-regular planar graph with $n$ vertices, there are at most $(2+|T|)^n$ multi-virtual knot diagrams $D$ with shadow $G$, since for every vertex we must decide whether the corresponding crossing is classical (overpass or underpass) or virtual of a type from $T$.
\end{proof}

\begin{corollary}\label{Cr:MVBound}
Let $(K_1,K_2,\dots)$ be an infinite family of pairwise nonequivalent multi-virtual knots. Then the sequence $(\mathrm{cr}_t(K_1),\mathrm{cr}_t(K_2),\dots)$ is not bounded.
\end{corollary}

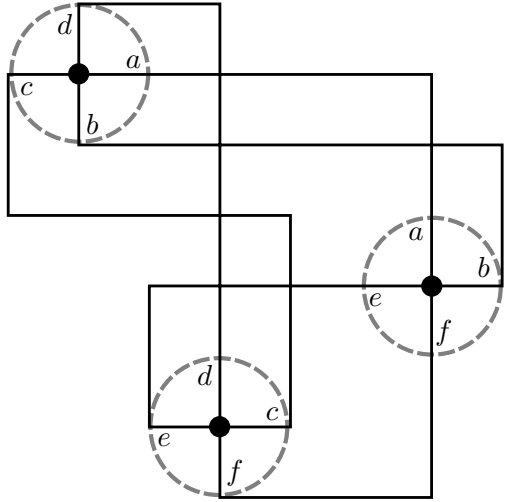
\begin{figure}[!ht]
\centering
\begin{overpic}[width = 0.4\textwidth]{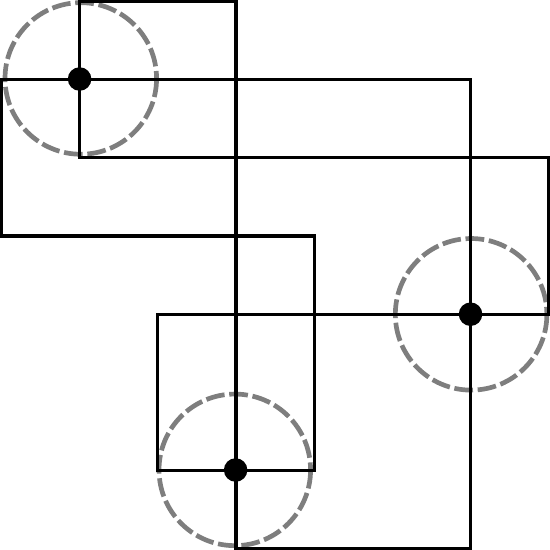}
\put(98,30){$c$}
\put(72,43){$d$}
\put(57,20){$e$}
\put(83,7){$f$}
\put(178,84){$b$}
\put(152,98){$a$}
\put(137,73){$e$}
\put(162,60){$f$}
\put(45,163){$a$}
\put(19,176){$d$}
\put(5,153){$c$}
\put(30,138){$b$}
\end{overpic}
\caption{A virtual knot as a $4$-regular graph. The dots represent classical crossings. Virtual crossings arise at the remaining intersections of the arcs.}\label{Fg:VirtualGraph}
\end{figure}

We will now show that for a virtual knot, the total crossing number is bounded above as a function of the classical crossing number. We give a short proof since we do not aim for the best possible upper bound.

\begin{prop}\label{Pr:VirtualGraphUB}
Let $K$ be a virtual knot with $\mathrm{cr}_c(K)=n$. Then $\mathrm{cr}_t(K)\le 6n^2-2n$.
\end{prop}
\begin{proof}
In one possible interpretation, a virtual knot is a connected $4$-regular graph $G$ that is drawn in the plane but is not necessarily planar. The vertices of $G$ correspond to classical crossings and the intersections of edges correspond to virtual crossings. Figure \ref{Fg:VirtualGraph} illustrates a construction of all virtual knots. Place $n$ vertices in the plane so that no two vertices lie on the same vertical or horizontal line. At each vertex $v$, indicate the four edges at $v$ by drawing four short labeled line segments emanating from $v$ in the four cardinal directions, keeping the short line segments within a small neighborhood of $v$. Connecting any pair of matching (that is, with the same label) short line segments from the edge of one neighborhood to the edge of the other neighborhood can be done with a curve consisting of at most $3$ line segments oriented in the cardinal directions and such that the direction of travel is never reversed (that is, it is not necessary to travel to the left and later to the right, or up and later down). As an illustration, consider the curve connecting the two line segments labeled $c$ in Figure \ref{Fg:VirtualGraph}. By shrinking the neighborhoods as needed, we can further arrange the situation so that the connecting curves avoid all vertex neighborhoods. Any two of the $2n$ connecting curves intersect at most $3$ times, since each curve consists of at most $3$ line segments in cardinal directions and the direction of travel is not reversed. Hence there are at most $0+3+2\cdot 3 + \cdots + (2n-1)\cdot 3 = 3n(2n-1)=6n^2-3n$ intersections. Adding the original $n$ vertices (classical crossings), we have our result.
\end{proof}

\begin{problem}
Improve the upper bound in Proposition \ref{Pr:VirtualGraphUB}.
\end{problem}

Combining Propositions \ref{Pr:MVBound} and \ref{Pr:VirtualGraphUB}, we obtain:

\begin{corollary}\label{Cr:VirtualGraph}
Up to equivalence, there are only finitely many virtual knots with $n$ classical crossings.
\end{corollary}

In the following example, we construct an infinite family of pairwise nonequivalent multi-virtual links with only two types of virtual crossings and with no classical crossings. See Subsection \ref{Ss:SingleClassical} for an infinite family of pairwise nonequivalent multi-virtual knots, each with a single classical crossing.

\def\ll{\left\langle}
\def\rr{\right\rangle}

First, we recall the chromatic bracket polynomial from \cite{Kau2}. To every multi-virtual link diagram $L$ with at most two types of virtual crossings, denoted by $\vcenter{ \hbox{\begin{overpic}[scale=0.2]{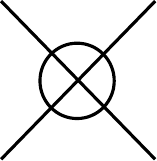}
    \end{overpic}}}$ and $\vcenter{ \hbox{\begin{overpic}[scale=0.2]{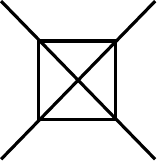}
    \end{overpic}}},$ we associate a Laurent polynomial $\ll L \rr \in \mathbb{Z}[A^{\pm1}]$ obtained by the skein relations\smallskip
\begin{enumerate}
\item[(i)]
    $\ll \vcenter{ \hbox{\begin{overpic}[scale=0.2]{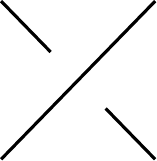}
    \end{overpic}}} \rr = A \ll \vcenter{ \hbox{\begin{overpic}[scale=0.2]{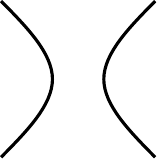}
    \end{overpic}}} \rr + A^{-1} \ll \vcenter{ \hbox{\begin{overpic}[scale=0.2]{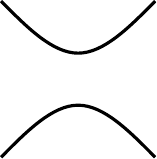}
    \end{overpic}}} \rr ,$\smallskip
\item[(ii)]
    $\ll L \sqcup O \rr = d \ll L \rr,$ where $O$ denotes a trivial component, and\smallskip
\item[(iii)]
    $ \ll \vcenter{ \hbox{\begin{overpic}[scale=0.2]{boxcrossing.pdf}
    \end{overpic}}} \rr = 2 \ll \vcenter{ \hbox{\begin{overpic}[scale=0.2]{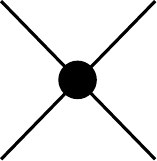}
    \end{overpic}}} \rr - \ll \vcenter{ \hbox{\begin{overpic}[scale=0.2]{circlecrossing.pdf}
    \end{overpic}}} \rr$ (the box-circle relation), \smallskip
\end{enumerate}
where the $\vcenter{ \hbox{\begin{overpic}[scale=0.2]{nodecrossing.pdf}
    \end{overpic}}}$ crossing in relation (iii) is referred to as a \emph{node}. Additionally,
\begin{enumerate}
\item[(iv)] all flat diagrams with only nodes are evaluated as $d = - A^2 - A^{-2}$.
\end{enumerate}
It then follows (see \cite{Kau2}), that the chromatic bracket polynomial is invariant under all multi-virtual Reidemeister moves except the move $\Omega1.$ We proceed to the example.

\begin{example}
Consider the link diagram $L(i,j)$ shown below. The first $i$ pairs of crossings are nodes followed by $j$ pairs of virtual crossings of box type and circle type.
\begin{displaymath}
    L(i,j)= \underbrace{\vcenter{\hbox{
        \begin{overpic}[scale=0.5]{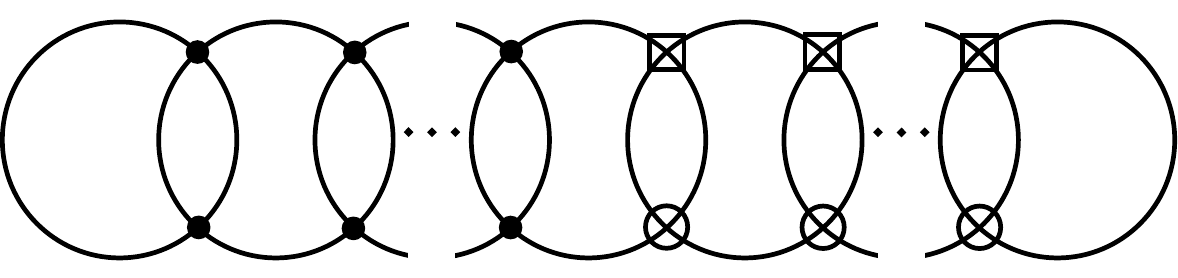}
        \end{overpic}}}}_{\text{$2i$ nodes, $j$ crossings of box type and $j$ crossings of circle type}}
\end{displaymath}
We will compute the chromatic bracket polynomial of the multi-virtual links $L(0,n)$. Using the box-circle relation and the virtual Reidemeister 2 move, we obtain: 
\begin{align*}
    \ll L(0,n)\rr &= 2\ll L(1, n - 1)\rr - d\ll L(0, n - 1)\rr\\
    &= 4\ll L(2,n - 2)\rr - 4d\ll L(1,n - 2)\rr + d^2\ll L(0,n - 2)\rr  \\
    &= 8\ll L(3,n - 3)\rr - 12d\ll L(2,n - 3)\rr + 6d^2\ll L(1,n - 3)\rr - d^3\ll L(0,n - 3)\rr\\
    &\ \,\vdots\\
    &= 2^n\ll L(n,0)\rr + \cdots + (-1)^nd^n\ll L(0,0)\rr\\
    &= 2^n\ll L(n,0)\rr  + \cdots + (-1)^nd^{n+1}.
\end{align*}
Observe that the highest degree term in $\ll L(0,n)\rr$ is $A^{2n + 2}$. This implies that the links in the infinite family $\{L(0,n):n\ge 0\}$ are distinct.
\end{example}

\section{Classifying multi-virtual links with given virtual projections}\label{Sc:Bell}

For a multi-virtual link diagram $D$, let $v(D)$ be the \emph{virtual projection} of $D$, the virtual link diagram obtained from $D$ by forgetting about the type of each virtual crossing. We can similarly define virtual projections of multi-virtual Reidemeister moves. Note that the virtual projection of a multi-virtual Reidemeister move is a virtual Reidemeister move. We therefore have:

\begin{prop}\label{Pr:Projection}
Let $D_1$ and $D_2$ be equivalent multi-virtual link diagrams. Then $v(D_1)$ and $v(D_2)$ are equivalent virtual link diagrams.
\end{prop}

By the contrapositive of Proposition \ref{Pr:Projection}, if $v(D_1)$ and $v(D_2)$ are not equivalent, then $D_1$ and $D_2$ are not equivalent. But it is a nontrivial question to decide whether $D_1$ and $D_2$ are equivalent when $v(D_1)$ and $v(D_2)$ are equivalent or even the same. For instance:

\begin{problem}
Given three distinct types $\alpha,\beta,\gamma$, are the two multi-virtual trefoils in Figure \ref{Fg:Trefoils} equivalent? (The virtual trefoil on the left in Figure \ref{Fg:Trefoils} is of course equivalent to the unknot, so we are asking if the multi-virtual trefoil with three distinct types of virtual crossings is equivalent to the unknot.)
\end{problem}

\begin{figure}[!ht]
\centering
\begin{overpic}[width = 0.6\textwidth]{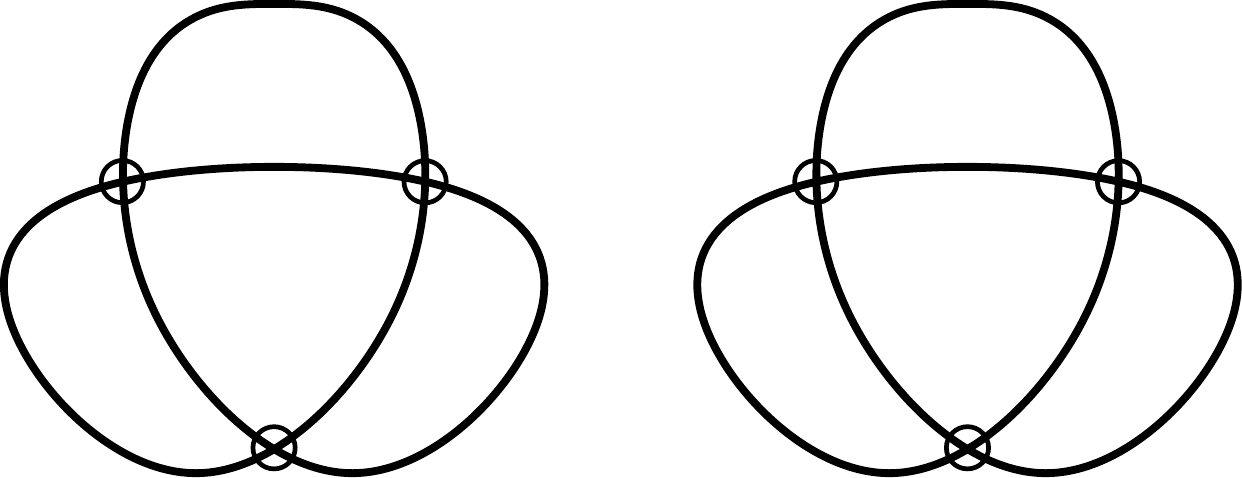}
\put(58,16){$\alpha$}
\put(33,55){$\alpha$}
\put(84,55){$\alpha$}
\put(191,55){$\alpha$}
\put(240,53){$\beta$}
\put(216,18){$\gamma$}
\end{overpic}
\caption{Are these multi-virtual trefoils equivalent?}\label{Fg:Trefoils}
\end{figure}
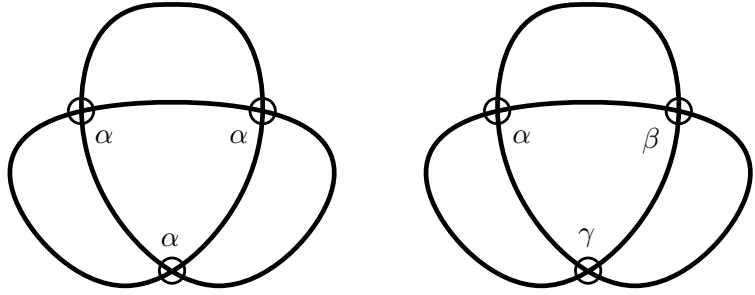

Let us now consider the inverse image of the mapping $v$. For a virtual link diagram $V$, let 
\begin{displaymath}
    m(V) = \{D:v(D) = V\} = v^{-1}(V)
\end{displaymath}
be the set of all multi-virtual link diagrams whose virtual projection is equal to $V$. We would like to understand the elements of $m(V)$ up to equivalence.

\subsection{Permutational equivalence of multi-virtual link diagrams}

Let $f:T\to T$ be a bijection and let $D$ be a multi-virtual link diagram. Then $f(D)$ is the diagram obtained from $D$ by renaming all types of virtual crossings according to $f$. In this context, $f$ will be referred to as a \emph{retyping}.

We say that two multi-virtual link diagrams are \emph{permutationally equivalent} if one is obtained from the other by a finite sequence of retypings, isotopy and multi-virtual Reidemeister moves.

It is not difficult to see that a retyping followed by a multi-virtual Reidemeister move can also be accomplished by a suitable multi-virtual Reidemeister move followed by a retyping. Hence two multi-virtual link diagrams $D_1$ and $D_2$ are permutationally equivalent if and only if there exists a bijection $f:T\to T$ such that $f(D_1)$ are $D_2$ are equivalent.

Clearly, any two equivalent diagrams are also permutationally equivalent. The converse does not hold in general but it might be difficult to establish this in a concrete case. 

\subsection{Multi-virtual link diagrams with the same virtual projection}\label{Ss:SameVirtual}

Let us return to the question of classifying all multi-virtual diagrams $m(V)$ of a given virtual diagram $V$ up to equivalence and up to permutational equivalence.

Let $B_n$ be the $n$th Bell number, the number of partitions of an $n$-element set. Recall that $B_0=0$ and $B_{n+1}=\sum_{k=0}^n\binom{n}{k}B_k$ for every $n\ge 0$.

\begin{prop}
Let $V$ be a virtual link diagram with $n$ virtual crossings. Then the set $m(V)$ contains at most $B_n$ multi-virtual link diagrams up to permutational equivalence.
\end{prop}
\begin{proof}
Let $C=\{c_1,\dots,c_n\}$ be the virtual crossings of $V$. Assigning types to all virtual crossings results in a partition of $C$. The resulting diagram can be retyped greedily by assigning type $\alpha_1$ to $c_1$ and to all virtual crossings of the same type as $c_1$, $\alpha_2$ to the first virtual crossing $c_i$ not yet visited and to all virtual crossings of the same type as $c_i$, and so on. 
\end{proof}

\begin{example}
Consider the virtual figure-eight knot diagram in Figure \ref{Fg:FigureEight}, where the labels $1$, $2$, $3$ are used to keep track of the virtual crossings, not to denote their types.

\begin{figure}[!ht]
\centering
\begin{overpic}[width = 0.15\textwidth]{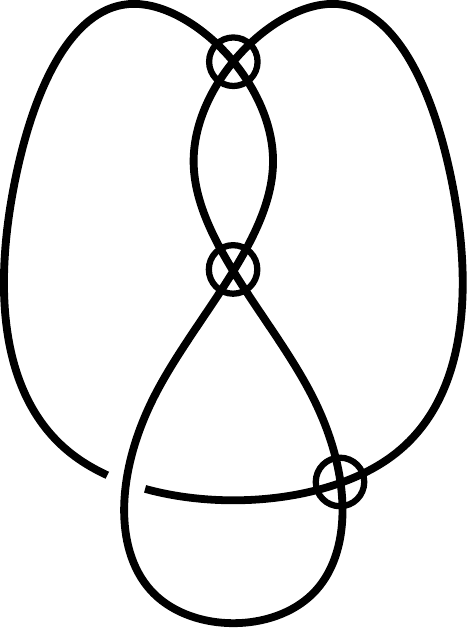}
\put(56,13){$3$}
\put(24,50){$2$}
\put(41,80){$1$}
\end{overpic}
\caption{A virtual figure-8 knot with virtual crossings labeled for reference.}\label{Fg:FigureEight}
\end{figure}

There are five partitions of the $3$-element set $\{1,2,3\}$:
\begin{displaymath}
    \{\{1,2,3\}\},\ \{\{1,2\},\{3\}\},\  \{\{1,3\},\{2\}\},\  \{\{1\},\{2,3\}\}\text{ and } \{\{1\},\{2\},\{3\}\}.
\end{displaymath} 
We can also represent the five partitions as the columns of the array
\begin{displaymath} \left[
    \begin{array}{ccccc}
    1&1&1&1&1\\
    1&1&2&2&2\\
    1&2&1&2&3
    \end{array}\right].
\end{displaymath}
In general, for partitions of an $n$-element set, the columns of the array are all sequences $a_1,\dots,a_n$ such that $a_1=1$ and $a_{i+1}\le \max\{a_1,\dots,a_i\}+1$ for every $1\le i<n$.

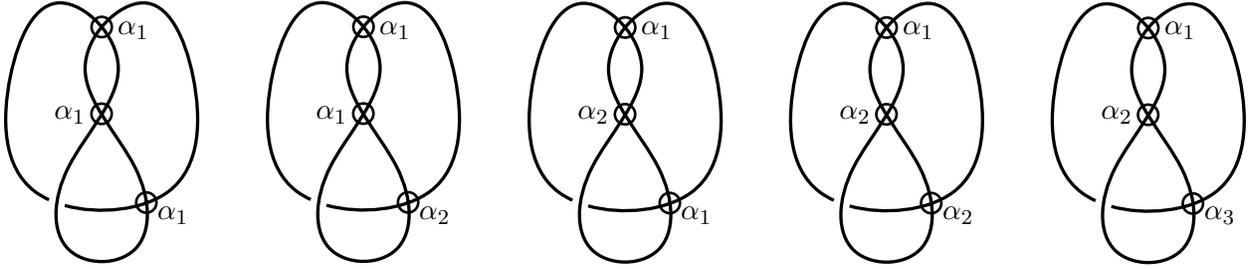
\begin{figure}[!ht]
\centering
\begin{overpic}[width = \textwidth]{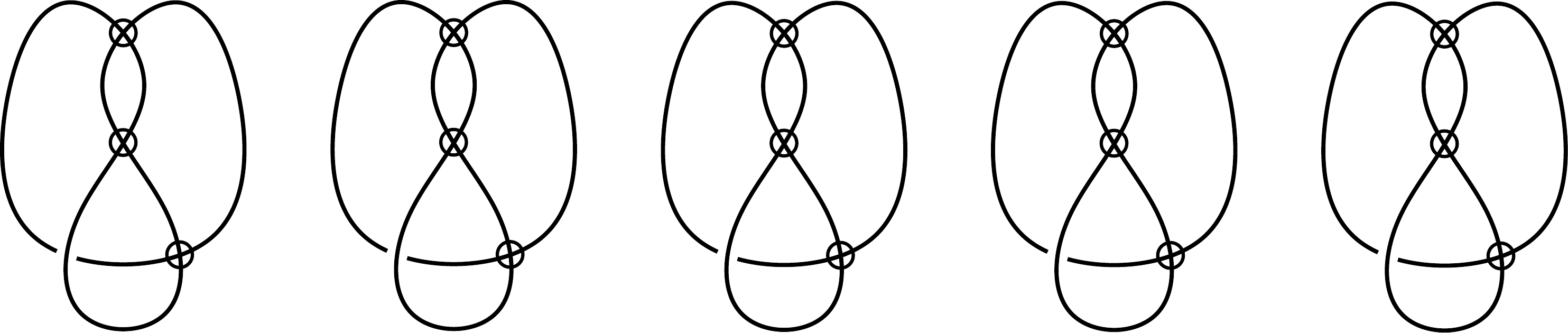}
\put(58,17){$\alpha_1$}
\put(19,55){$\alpha_1$}
\put(43,87){$\alpha_1$}
\put(157,17){$\alpha_2$}
\put(118,55){$\alpha_1$}
\put(142,87){$\alpha_1$}
\put(256,17){$\alpha_1$}
\put(217,55){$\alpha_2$}
\put(241,87){$\alpha_1$}
\put(355,17){$\alpha_2$}
\put(316,55){$\alpha_2$}
\put(340,87){$\alpha_1$}
\put(454,17){$\alpha_3$}
\put(415,55){$\alpha_2$}
\put(439,87){$\alpha_1$}
\end{overpic}
\caption{The multi-virtual family of the virtual knot diagram in Figure \ref{Fg:FigureEight} up to permutational equivalence. It is conceivable that some of the diagrams are still (permutationally) equivalent.}\label{Fg:FigureEights}
\end{figure}

The resulting five multi-virtual diagrams can be seen in Figure \ref{Fg:FigureEights}. No two of the five diagrams are obtained from each other just by retyping, and it is guaranteed that every diagram of $m(V)$ is permutationally equivalent to at least one of the diagrams in the figure. But it is conceivable that some of the diagrams in Figure \ref{Fg:FigureEights} are (permutationally) equivalent.
\end{example}

Working only up to equivalence (rather than up to permutational equivalence) is more subtle. Let $V_1$, $V_2$ be two virtual link diagrams with $n_1$ and $n_2$ virtual crossings, respectively, and let $D_1\in m(V_1)$, $D_2\in m(V_2)$. The diagrams $D_1$ and $D_2$ might have some types in common. While studying the equivalence of $D_1$ and $D_2$, we should therefore not retype the diagrams $D_1$ and $D_2$ independently. However, we can retype the union $D_1\cup D_2$ of the two diagrams. In this way we obtain $B_{n_1+n_2}$ pairs of diagrams for which the equivalence problem needs to be solved. For instance, when $n_1=n_2=2$, we obtain $B_4=15$ pairs of diagrams with types assigned according to the columns of
\begin{equation}\label{Eq:2Plus2}
    \begin{array}{ccccccccccccccc}
    1&1&1&1&1&1&1&1&1&1&1&1&1&1&1\\
    1&1&1&1&1&2&2&2&2&2&2&2&2&2&2\\
    \hline
    1&1&2&2&2&1&1&1&2&2&2&3&3&3&3\\
    1&2&1&2&3&1&2&3&1&2&3&1&2&3&4
    \end{array}\ ,
\end{equation}
where the first $n_1$ rows correspond to the virtual crossings of $V_1$ and the last $n_2$ rows correspond to the virtual crossings of $V_2$.

\begin{figure}[!ht]
\centering
\begin{overpic}[width = 0.6\textwidth]{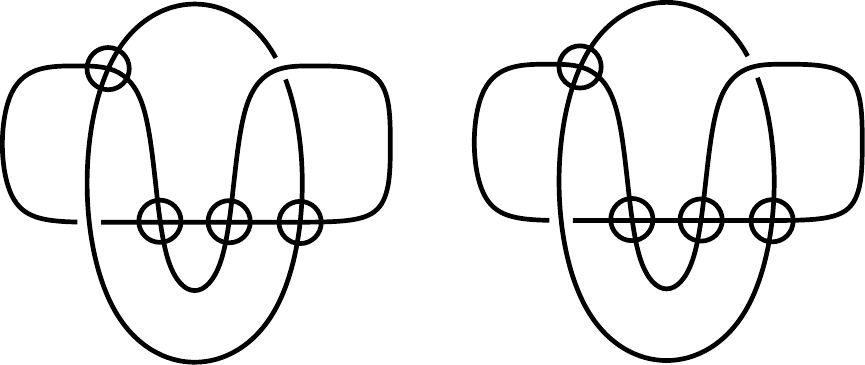}
\put(25,105){$\alpha$}
\put(40,32){$\beta$}
\put(63,34){$\alpha$}
\put(86,32){$\beta$}
\put(180,105){$\beta$}
\put(195,34){$\alpha$}
\put(216,32){$\beta$}
\put(240,34){$\alpha$}
\end{overpic}
\caption{Are these permutationally equivalent links equivalent?}\label{Fg:LinksPair}
\end{figure}
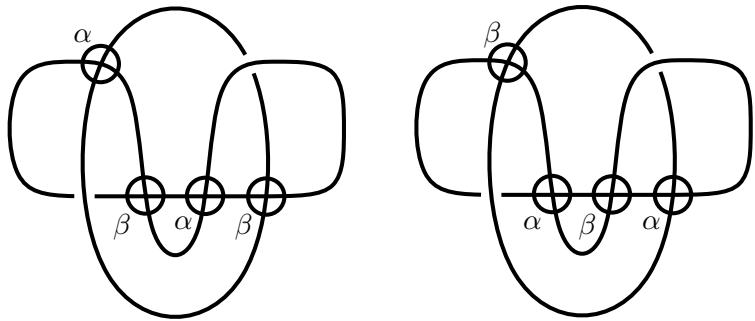

Already the case $V_1=V_2$ is interesting. To illustrate the equivalence problem in that case, Figure \ref{Fg:LinksPair} gives two diagrams with the same virtual projection $V$, each with four virtual crossings. The two diagrams are obviously permutationally equivalent, but are they equivalent? The answer turns out to be negative---see Example \ref{Ex:Special}. This is one of the $B_{4+4}=B_8=4140$ pairs of diagrams in $m(V)$ that need to be investigated for equivalence in order to fully understand the family $m(V)$ of diagrams up to equivalence. 

\section{Invariants of multi-virtual links based on quandle colorings}\label{Sc:Invariants}

Quandles are the algebras of knot theory \cite{EN,Joy1,Mat} whose axioms arise naturally from the classical Reidemeister moves. We will give a definition of quandles that has a more algebraic flavor than is usually the case in low-dimensional topology, but which is nevertheless equivalent.

A set $Q$ with a binary operation $*$ is called a \emph{magma}. Given a magma $(Q,*)$ and $x\in Q$, the \emph{right translation} $R_x:Q\to Q$ is defined by $R_x(y)=y*x$.

A magma $(Q,*)$ is a \emph{right quasigroup} if for every $x\in Q$ the right translation $R_x$ is a permutation of $Q$. We will then denote by $/$ the \emph{right division} operation on $Q$ defined by $x/y = R_x^{-1}(y)$, and we note that $(y*x)/x = y = (y/x)*x$ holds for all $x,y\in Q$.

The \emph{right multiplication group} of a right quasigroup $(Q,*)$ is the permutation group
\begin{displaymath}
    \rmlt{Q,*} = \langle R_x:x\in Q\rangle
\end{displaymath}
generated by all right translations of $(Q,*)$.

An \emph{endomorphism} of a right quasigroup $(Q,*)$ is a mapping $f:Q\to Q$ such that $f(x*y) = f(x)*f(y)$ for all $x,y\in Q$. Note that then $f(x/y) = f(x)/f(y)$ also holds, since $f(x/y)*f(y) = f((x/y)*y)=f(x)$. A bijective endomorphism is an \emph{automorphism}. The automorphisms of $(Q,*)$ form a group under composition, the \emph{automorphism group} $\aut{Q,*}$ of $(Q,*)$.

A right quasigroup $(Q,*)$ is a \emph{rack} if it satisfies \emph{right self-distributivity}
\begin{displaymath}
    (x*y)*z = (x*z)*(y*z).
\end{displaymath}
In other words, a rack is a magma in which every right translation is an automorphism. A \emph{quandle} is a rack $(Q,*)$ that satisfies the \emph{idempotent} identity
\begin{displaymath}
    x*x=x.
\end{displaymath}

\subsection{Operator quandle colorings} 

The idea of using a quandle automorphism in the coloring rule at a virtual crossing first appeared in a paper of Manturov \cite{Man} and was further developed in \cite{KM}. In \cite{Kau2}, the first author defined multi-virtual quandles and uses it to distinguish knots and purely multi-virtual diagrams. In the following subsections, we use multi-virtual quandles to construct invariants for multi-virtual links.

An \emph{operator quandle} (or a \emph{multi-virtual quandle}) is a quandle $(Q,*)$ together with a list $A$ of pairwise commuting automorphisms of $(Q,*)$. We will be assigning automorphisms from the list $A$ to virtual crossings, with all virtual crossings of the same type receiving the same automorphism. For multi-virtual link diagrams of type $T$, we will generically denote the list of automorphisms by $A_T$ and we will denote the automorphism of $(Q,*)$ assigned to crossings of type $\alpha$ also just by $\alpha$.

\begin{figure}[!ht]
\centering
\begin{overpic}[width = 0.6\textwidth]{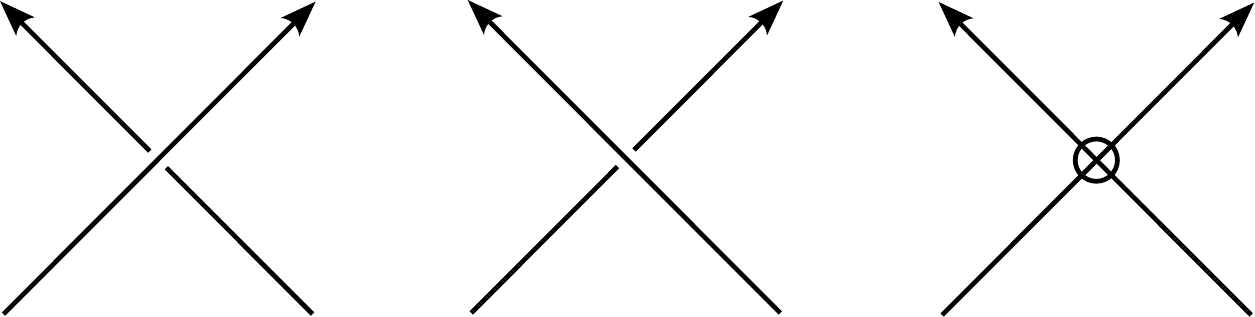}
\put(5,15){$b$}
\put(60,15){$a$}
\put(-8,48){$a*b$}
\put(60,48){$b$}
\put(110,15){$b$}
\put(165,15){$a$}
\put(110,48){$a$}
\put(165,48){$b/a$}
\put(216,15){$b$}
\put(271,15){$a$}
\put(254,33){$\alpha$}
\put(200,48){$\alpha(a)$}
\put(270,48){$\alpha^{-1}(b)$}
\end{overpic}
\caption{Operator quandle coloring rules for multi-virtual link diagrams.}\label{Fg:ColoringRules}
\end{figure}
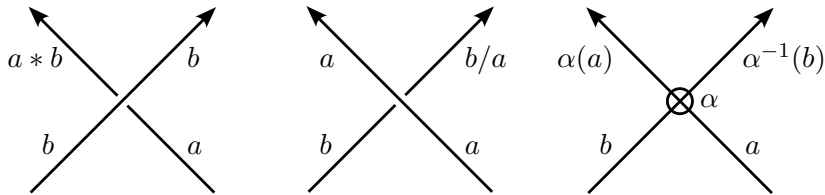

Let $D$ be an oriented multi-virtual link diagram with virtual crossings of type $T$. Let $(Q,*,A_T)$ be an operator quandle. Then an \emph{operator quandle coloring}, or just \emph{coloring}, of $D$ is an assignment of elements of $Q$ to the arcs of $D$ compatible with the coloring rules depicted in Figure \ref{Fg:ColoringRules}. We let $\col{D,Q,*,A_T}$ denote the number of colorings of $D$ by the operator quandle $(Q,*,A_T)$.

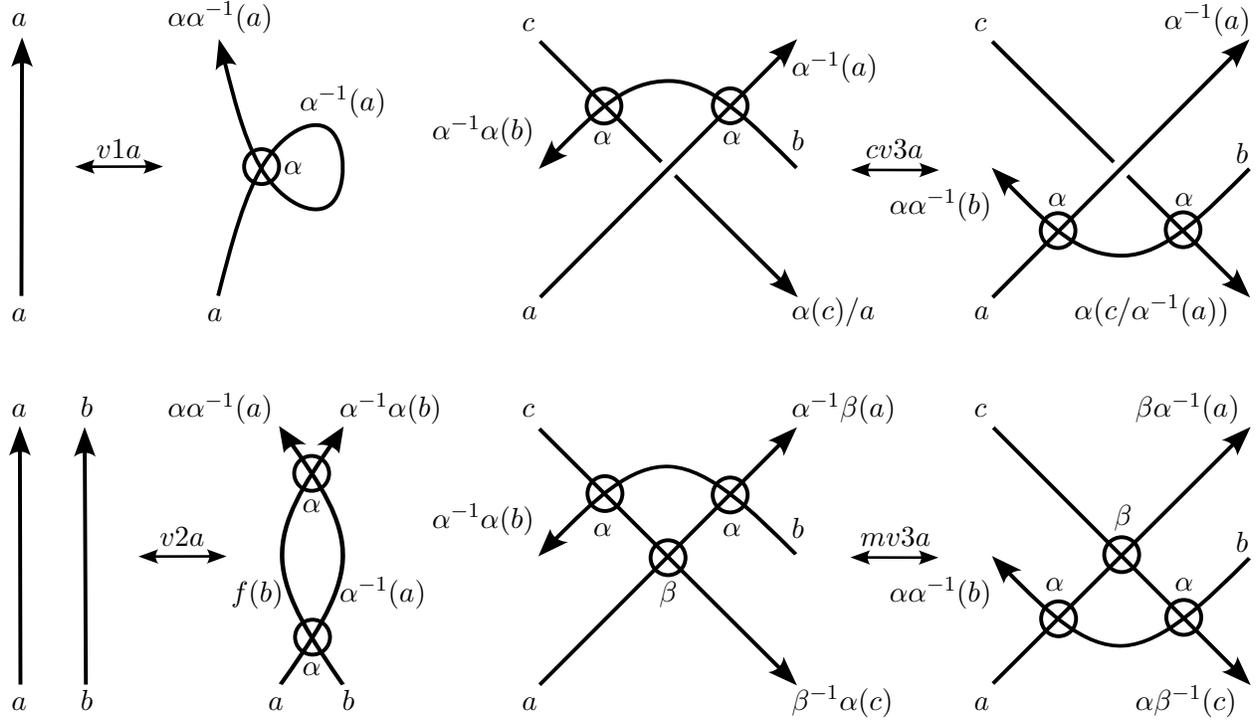
\begin{figure}[!ht]
\bigskip
\centering

\begin{overpic}[width = \textwidth]{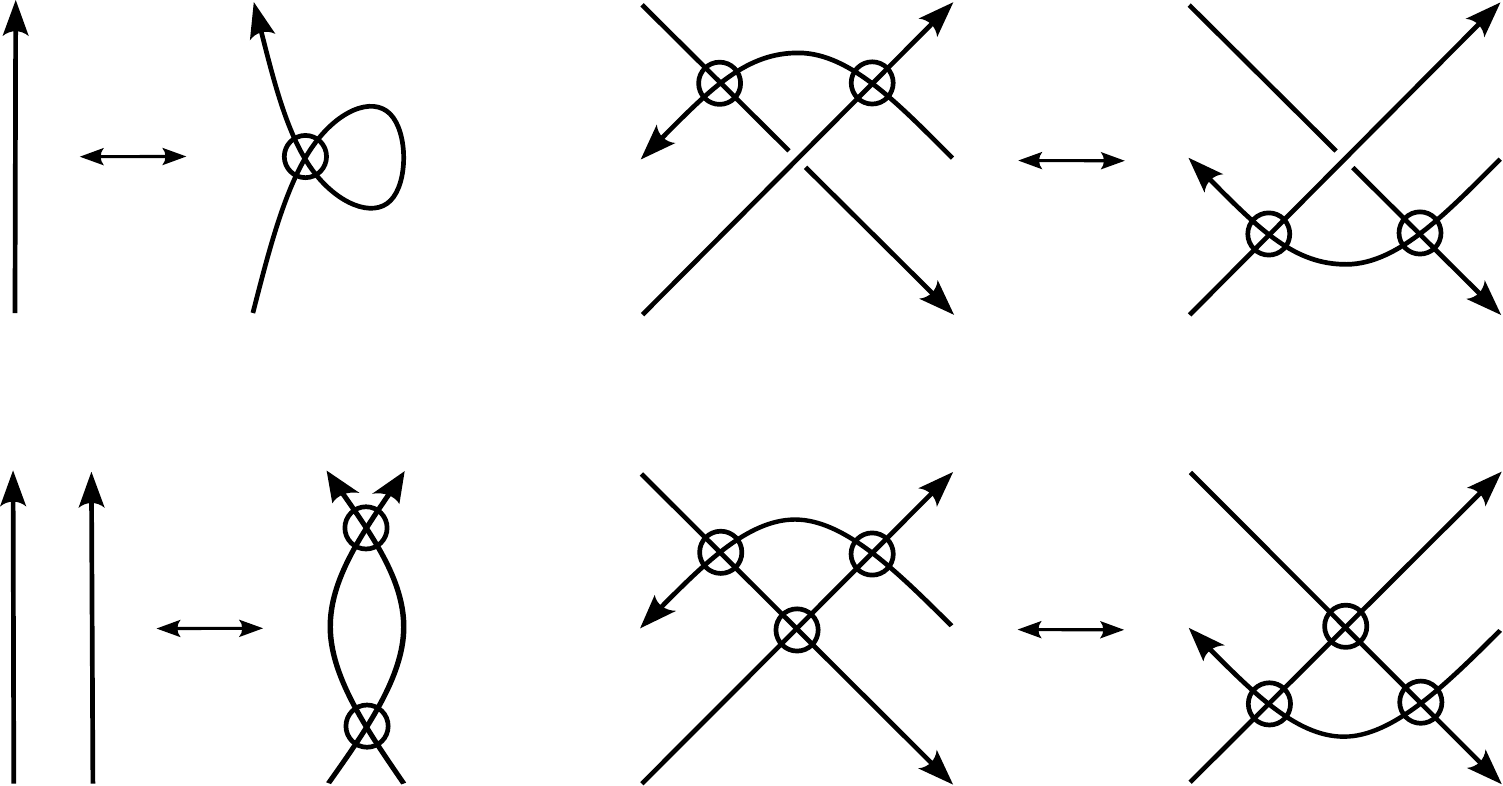}

\put(1,138){$a$}
\put(1,250){$a$}
\put(75,138){$a$}
\put(60,250){$\alpha\alpha^{-1}(a)$}
\put(110,218){$\alpha^{-1}(a)$}
\put(33,200){$v1a$}
\put(104,194){$\alpha$}

\put(1,-9){$a$}
\put(1,102){$a$}
\put(27,-9){$b$}
\put(27,102){$b$}
\put(98,-9){$a$}
\put(126,-9){$b$}
\put(84,32){$f(b)$}
\put(125,32){$\alpha^{-1}(a)$}
\put(60,102){$\alpha\alpha^{-1}(a)$}
\put(125,102){$\alpha^{-1}\alpha(b)$}
\put(111,4){$\alpha$}
\put(111,66){$\alpha$}
\put(57,53){$v2a$}

\put(194,139){$a$}
\put(296,139){$\alpha(c)/a$}
\put(296,203){$b$}
\put(160,207){$\alpha^{-1}\alpha(b)$}
\put(194,248){$c$}
\put(296,231){$\alpha^{-1}(a)$}
\put(221,205){$\alpha$}
\put(270,205){$\alpha$}

\put(365,139){$a$}
\put(403,139){$\alpha(c/\alpha^{-1}(a))$}
\put(464,198){$b$}
\put(333,179){$\alpha\alpha^{-1}(b)$}
\put(365,248){$c$}
\put(437,249){$\alpha^{-1}(a)$}
\put(393,182){$\alpha$}
\put(441,182){$\alpha$}

\put(324,200){$cv3a$}

\put(194,-9){$a$}
\put(296,-9){$\beta^{-1}\alpha(c)$}
\put(296,56){$b$}
\put(160,60){$\alpha^{-1}\alpha(b)$}
\put(194,102){$c$}
\put(296,102){$\alpha^{-1}\beta(a)$}
\put(246,32){$\beta$}
\put(221,57){$\alpha$}
\put(270,57){$\alpha$}

\put(365,-9){$a$}
\put(426,-9){$\alpha\beta^{-1}(c)$}
\put(464,51){$b$}
\put(333,32){$\alpha\alpha^{-1}(b)$}
\put(365,102){$c$}
\put(426,102){$\beta\alpha^{-1}(a)$}
\put(418,61){$\beta$}
\put(392,36){$\alpha$}
\put(441,36){$\alpha$}

\put(322,53){$mv3a$}
\end{overpic}
\bigskip
\caption{Compatibility of multi-virtual Reidemeister moves with operator quandle colorings rules.}\label{Fg:ColoringCheck}
\end{figure}

\begin{theorem}\label{Th:Coloring}
Let $D_1$ and $D_2$ be two equivalent oriented multi-virtual link diagrams with virtual crossings typed by $T$. Let $(Q,*,A_T)$ be an operator quandle. Then $\col{D_1,Q,*,A_T} = \col{D_2,Q,*,A_T}$.
\end{theorem}
\begin{proof}
It is well known that the two quandle coloring rules for classical crossings depicted in Figure \ref{Fg:ColoringRules} are compatible with classical Reidemeister moves, cf.~\cite{EN}. It therefore suffices to prove that all three coloring rules of Figure \ref{Fg:ColoringRules} are compatible with multi-virtual Reidemeister moves. By Theorem \ref{Th:OrientedBasis}, it suffices to check the five moves in Figure \ref{Fg:Minimalmoves} that involve virtual crossings. This is done in Figure \ref{Fg:ColoringCheck}, where we list only the moves $v1a$, $v2a$, $cv3a$ and $mv3a$ because $v1b$ behaves as $v1a$ for the purposes of operator quandle colorings. The moves $v1a$ and $v2a$ yield the trivial condition $\alpha\alpha^{-1}=1=\alpha^{-1}\alpha$. From $cv3a$ we get $\alpha^{-1}\alpha = \alpha\alpha^{-1}$ and $\alpha(c)/b = \alpha(c/\alpha^{-1}(b))$, where the latter condition holds because $\alpha$ is an automorphism of $(Q,*)$ and hence also an automorphism with respect to the right division operation. Finally, $mv3a$ again yields the trivial condition $\alpha^{-1}\alpha = \alpha\alpha^{-1}$ and also $\alpha^{-1}\beta = \beta\alpha^{-1}$ and $\beta^{-1} \alpha = \alpha \beta^{-1}$, which hold because the automorphisms from $A_T$ pairwise commute.
\end{proof}

Let $L$ be an oriented multi-virtual link and $(Q,*,A_T)$ an operator quandle. Then $\col{L,Q,*,A_T}$ is defined as $\col{D,Q,*,A_T}$, where $D$ is any oriented diagram of $L$. Theorem \ref{Th:Coloring} guarantees that $\col{L,Q,*,A_T}$ is well-defined and is therefore an invariant of $L$, called the \emph{operator quandle coloring invariant} of $L$ with respect to $(Q,*,A_T)$.

\subsection{A 2-cocycle invariant for multi-virtual links}

In this subsection, generalizing \cite{CJKLS, Kaz}, we introduce an invariant of oriented multi-virtual links based on $2$-cocycles of operator quandles.

Given a quandle $(Q,*)$ and a group $(G,\cdot)$, a mapping $\phi:Q\times Q\to G$ is a \emph{quandle $2$-cocycle} of $(Q,*)$ with values in $G$ if for all $a,b,c \in Q$
\begin{itemize}

\item[(i)] $\phi(a,a) = 0$ and

\item[(ii)] $\phi(a,b) + \phi(a*c,b*c) = \phi(a,c) + \phi(a*b,c).$

\end{itemize}

Let $(Q,*,A_T)$ be an operator quandle and $(G,\cdot)$ a group. Then a mapping $\phi:Q\times Q\to G$ is an \emph{operator quandle $2$-cocycle} of $(Q,*,A_T)$ with values in $G$ if it is a quandle $2$-cocycle of $(Q,*)$ that satisfies the condition 
\begin{equation}\label{Eq:OperatorCocycle}
    \phi(a,b) = \phi(\alpha(a),\alpha(b))
\end{equation}
for all $a,b \in Q$ and $\alpha \in A_T$.

Let $(Q,*,A_T)$ be an operator quandle, $(G,\cdot)$ a group with identity element $1$ and $\phi:Q\times Q\to G$ an operator quandle $2$-cocycle. Let $D$ be an oriented multi-virtual link diagram. Given a classical crossing $c$ of $D$ and an operator quandle coloring $C$ of $D$ by elements of $(Q,*,A_T)$, let $\phi(a,b)^\epsilon$ be the weight assigned to $c$ with respect to $C$, where $a$ and $b$ are the colors of $C$ on the incoming arcs of $c$, $\epsilon = 1$ if $c$ is a positive crossing and $\epsilon = -1$ if $c$ is a negative crossing. Let
\begin{equation}\label{Eq:CocInvariant}
    \mathrm{Coc}(D,Q,*,A_T,\phi) = \sum_C\prod_c\phi(a,b)^\epsilon,
\end{equation}
where the sum is taken over all operator quandle colorings $C$ of $D$ by elements of $(Q,*,A_T)$, and the product is taken over all classical crossings $c$ of $D$. The value $\mathrm{Coc}(D,Q,*,A_T,\phi)$ should be seen as an element of the group ring $\mathbb Z G$. If $D$ has no classical crossings,  the empty product $\prod_c \phi(a,b)^\epsilon$ is equal to $1$. If $D$ has no operator quandle colorings by $(Q,*,A_T)$, we set the value of $\mathrm{Coc}(D,Q,*,A_T,\phi)$ to be equal to $0$.

\begin{theorem}\label{Th:Coc}
Let $D_1$ and $D_2$ be two equivalent oriented multi-virtual links. Let $(Q,*,A_T)$ be an operator quandle, $(G,\cdot)$ a group and $\phi:Q\times Q\to G$ an operator quandle $2$-cocycle. Then $\mathrm{Coc}(D_1,Q,*,A_T,\phi)=\mathrm{Coc}(D_2,Q,*,A_T,\phi)$.
\end{theorem}

\begin{proof}
It is well known that the quandle $2$-cocycle invariant is compatible with the classical Reidemeister moves \cite{CJKLS}. Among the remaining moves of Figure \ref{Fg:ColoringCheck}, only the $cv3a$ move involves a classical crossing. Let us fix a classical crossing and an operator quandle coloring that labels the arcs in the $cv3a$ move as shown in Figure \ref{Fg:ColoringCheck}. The contribution to the claimed invariant is $\phi(\alpha(c),a)^{-1}$ in the left diagram of $cv3a$ and $\phi(c,\alpha^{-1}(a))^{-1}$ in the right diagram of $cv3a$. By \eqref{Eq:OperatorCocycle}, the contributions coincide.
\end{proof}

For an oriented multi-virtual link $L$, an operator quandle $(Q,*,A_T)$ and an operator quandle $2$-cocycle $\phi:Q\times Q\to G$, let $\mathrm{Coc}(L,,Q,*,A_T,\phi)=\mathrm{Coc}(D,Q,*,A_T,\phi)$, where $D$ is any diagram of $L$. Theorem \ref{Th:Coc} guarantees that $\mathrm{Coc}(L,Q,*,A_T,\phi)$ is well-defined. This is the \emph{operator quandle $2$-cocycle invariant} for the oriented multi-virtual link $L$ with respect to $(Q,*,A_T)$ and $\phi$.

\begin{rem}
Other algebraic invariants for links could be generalized to the multi-virtual setting, for example, Yang-Baxter cocycle invariants \cite{CN1}, quandle quiver invariants \cite{CN2} and invariants based on the Zh-construction \cite{CT}.
\end{rem}

\section{Quandles with many pairwise commuting right translations}\label{Sc:Quandles}

In the definition of operator quandle colorings we did not demand that the automorphisms assigned to distinct virtual crossing types must be distinct. However, should two distinct types $\alpha$, $\beta$ be assigned the same automorphism, then certainly the operator quandle coloring invariant will not be able to distinguish between the two possibly nonequivalent multi-virtual links obtained from one another by interchanging the types $\alpha$ and $\beta$.

We are therefore interested in quandles with many pairwise commuting automorphisms. Since every right translation is an automorphism in a quandle, we want to understand quandles with many commuting right translations, the topic of this section. Readers interested only in topological considerations can take note of Proposition \ref{Pr:Construction} and skip ahead to Section \ref{Sc:Distinguish}.

Throughout this section we will denote the multiplication operation in a magma by $\cdot$ or by juxtaposition, rather than by $*$. If $Q$ is a magma, $x\in Q$ and $f\in\aut{Q}$ then
\begin{displaymath}
    R_{f(x)} = fR_xf\inv.
\end{displaymath}
In particular, if $Q$ is rack and $x,y\in Q$ then 
\begin{displaymath}
    R_{xy} = R_yR_xR_y\inv.
\end{displaymath}    

For a rack $Q$, let 
\begin{displaymath}
    \rdis{Q} = \langle R_x^{-1}R_y:x,y\in Q\rangle
\end{displaymath}
be the \emph{right displacement group} of $Q$. It is well known that $\rdis{Q}\unlhd\rmlt{Q}$ and $\rmlt{Q}/\rdis{Q}$ is cyclic \cite{Joy1,Joy2,HSV}.

\subsection{Racks with an abelian right multiplication group}

An extreme case of commuting right translations occurs when the entire right multiplication group is abelian. Let us characterize such racks and quandles.

A magma is \emph{medial} if it satisfies the identity 
\begin{displaymath}
    (xu)(vy) = (xv)(uy),
\end{displaymath}
\emph{graphic} if it satisfies the identity 
\begin{displaymath}
    (xy)x = xy,
\end{displaymath}
\emph{paragraphic} if it satisfies the identity
\begin{displaymath}
    x(yx) = xy,
\end{displaymath}
and \emph{flexible} if it satisfies the identity
\begin{displaymath}
    x(yx)=(xy)x.
\end{displaymath}

Medial quandles have been studied extensively \cite{JPSZ,JPZ,Tra}. In \cite{CMP}, a semigroup is said to be graphic if it satisfies the identity $x(yx)=yx$, which is the dual (mirror image) of the graphic identity given above. See \cite{MP} for homology of graphic quandles.
The tentative name ``paragraphic'' is new here and it is supposed to mimic the terminology for medial and paramedial groupoids \cite{CJK}. The flexible law is used extensively in quasigroup theory and nonassociative algebra \cite{Sch}.

\begin{prop}\label{Pr:RackAbelian}
The following conditions are equivalent for a rack $Q$:
\begin{enumerate}
    \item[(i)] $\rmlt{Q}$ is abelian,
    \item[(ii)] $(xy)z = (xz)y$ for all $x,y,z\in Q$,
    \item[(iii)] $x(yz) = xy$ for all $x,y,z\in Q$,
    \item[(iv)] $Q$ is medial and paragraphic.
\end{enumerate}
\end{prop}
\begin{proof}
Condition (ii) says that $R_yR_z=R_zR_y$ for all $y,z\in Q$ and hence (i) is equivalent to (ii). Condition (iii) says that $R_{yz}=R_y$ for all $y,z\in Q$. Since $R_{yz} = R_zR_yR_z^{-1}$, (iii) is equivalent to (ii).

Suppose that $\rmlt{Q}$ is abelian. Then $(xu)(vy) = (xu)v$ by (iii) and, similarly, $(xv)(uy) = (xv)u$. Since $(xu)v=(xv)u$, $Q$ is medial. Substituting $z=x$ into (iii) shows that $Q$ is paragraphic.

Finally, suppose that $Q$ is medial and paragraphic, and let us establish (iii). We have $xy=x(yx)$ since $Q$ is paragraphic, $x(yx)=((x/z)z)(yx) = ((x/z)y)(zx)$ by mediality, and $((x/z)y)(zx) = ((x/z)(zx))(y(zx))$ by right self-distributivity. So far we showed
\begin{equation}\label{Eq:Aux1}
    xy = ((x/z)(zx))(y(zx)).
\end{equation}
Now, $z(x/z) = z((x/z)z) = zx$ by the paragraphic law, and hence $(x/z)(zx)=(x/z)(z(x/z)) = (x/z)z = x$ again by the paragraphic law. Substituting this into \eqref{Eq:Aux1} yields $xy = x(y(zx))$ and replacing $z$ with $z/x$ then gives $xy = x(yz)$.
\end{proof}

\begin{example}
The finite model builder \texttt{mace4} \cite{McC} easily finds a paragraphic quandle of order $8$ that is not medial. Conversely, the $3$-element quandle $(\mathbb Z_3,\cdot)$ given by $x\cdot y = -x-y$ is medial but not paragraphic. Both properties are therefore required in condition (iv) of Proposition \ref{Pr:RackAbelian}.
\end{example}

\begin{lemma}\label{Lm:GraphicParagraphic}
Graphic racks, graphic quandles and paragraphic quandles are the same objects.
\end{lemma}
\begin{proof}
Note that every quandle is flexible: $(xy)x = (xx)(yx) = x(yx)$. Also note that the graphic and paragraphic axioms coincide in the presence of flexibility. It therefore suffices to prove that every graphic rack is idempotent. Substituting $x=y$ into the graphic axiom yields $(xx)x=xx$, and we obtain $xx=x$ upon canceling $x$ on the right.
\end{proof}

\begin{example}
The $2$-element rack $(\mathbb Z_2,\cdot)$ given by $x\cdot y = x+1$ is paragraphic but not graphic.
\end{example}

Combining Proposition \ref{Pr:RackAbelian} with Lemma \ref{Lm:GraphicParagraphic} yields:

\begin{corollary}
The following conditions are equivalent for a quandle $Q$:
\begin{enumerate}
    \item[(i)] $\rmlt{Q}$ is abelian,
    \item[(ii)] $Q$ is medial and paragraphic,
    \item[(iii)] $Q$ is medial and graphic.
\end{enumerate}
\end{corollary}

When $\rmlt{Q}$ is abelian then certainly the subgroup $\rdis{Q}$ is abelian. We recall:

\begin{prop}[{\cite[Prop. 2.4]{HSV}}]
Let $Q$ be a rack. Then $\rdis{Q}$ is abelian if and only if $Q$ is medial.
\end{prop}

\subsection{Connected racks with an abelian right multiplication group}

Recall that a rack $Q$ is \emph{connected} if the permutation group $\rmlt{Q}$ acts transitively on $Q$. A catalog of small connected quandles can be found in the \texttt{GAP} packages \texttt{rig} \cite{VG} and \texttt{RightQuasigroups} \cite{NV}. We will denote the $m$th connected quandle of order $n$ by \texttt{ConnectedQuandle(n,m)}.
 
It is customary to color knot diagrams by connected quandles rather than just by quandles. Let us have a look at connected quandles with many commuting right translations. We will first show that, as is well known, the extreme case of $\rmlt{Q}$ abelian cannot occur in nontrivial connected quandles.

For a group $G$, let us denote by $G'$ the derived subgroup of $G$.

\begin{prop}[{\cite[Prop. 2.3]{HSV}}]\label{Pr:ConnRack} Let $Q$ be a connected rack. Then $\rmlt{Q}'=\rdis{Q}$.
\end{prop}

A magma $Q$ is called \emph{permutational} if there is a permutation $f$ of $Q$ such that $xy=f(x)$ for all $x,y\in Q$. We denote this permutational magma by $Q(f)$. Note that every permutational magma is a rack since $(xz)(yz) = f(xz) = f^2(x) = f(xy) = (xy)z$ in $Q(f)$.

\begin{lemma}\label{Lm:Permutational}
The following conditions are equivalent for a magma $Q$:
\begin{enumerate}
\item[(i)] $Q$ is a connected rack and $\rmlt{Q}$ is abelian,
\item[(ii)] $Q=Q(f)$ is a permutational magma and $\langle f\rangle$ acts transitively on $Q$.
\end{enumerate}
\end{lemma}
\begin{proof}
Suppose that $Q$ is a connected rack and $\rmlt{Q}$ is abelian. Then Proposition \ref{Pr:ConnRack} yields $\rdis{Q}=\rmlt{Q}'=1$. Hence $R_x=R_y$ for all $x,y\in Q$ and $Q$ is permutational, say $Q=Q(f)$. The group $\langle f\rangle = \rmlt{Q}$ then acts transitively on $Q$.

Conversely, let $Q=Q(f)$ be a permutational magma (hence a rack) such that $\langle f\rangle$ acts transitively on $Q$. Then $\rmlt{Q}=\langle f\rangle$ is abelian and $Q$ is connected.
\end{proof}

\begin{prop}\label{Pr:ConnAndAb}
A magma $Q$ is a connected quandle with $\rmlt{Q}$ abelian if and only if $|Q|=1$.
\end{prop}
\begin{proof}
Let $Q$ be a connected quandle with $\rmlt{Q}$ abelian. By Lemma \ref{Lm:Permutational}, $Q=Q(f)$ is permutational and $\langle f\rangle$ acts transitively on $Q$. Since $f(x)=xx=x$ for all $x\in Q$, we have $\langle f\rangle = 1$ and $|Q|=1$. Conversely, if $|Q|=1$, $Q$ is clearly a connected quandle with $\rmlt{Q}$ abelian.
\end{proof}

\subsection{R-cliques}
Let us write $[x,y] = x^{-1}y^{-1}xy$ for the commutator of two group elements $x$ and $y$. Let $\sim$ be the binary relation on a right quasigroup $Q$ defined by
\begin{displaymath}
    a\sim b\text{ if and only if }[R_a,R_b]=1.
\end{displaymath}
If $a\sim b$, we say that $a$ and $b$ \emph{$R$-commute}. A subset $C$ of a right quasigroup $Q$ is an \emph{$R$-clique} if $a\sim b$ for every $a,b\in C$.

Note that every singleton subset of a right quasigroup is an $R$-clique, and every subset of a projection quandle (that is, a quandle with multiplication $xy=x$) is an $R$-clique. 

\begin{lemma}
Let $Q$ be a rack. Then $\sim$ is reflexive, symmetric and invariant under $\aut{Q}$.    
\end{lemma}
\begin{proof}
Reflexivity and symmetry of $\sim$ are clear. Let $f\in\aut{Q}$ and $a\sim b$ so that $[R_a,R_b]=1$. Then $[R_{f(a)},R_{f(b)}] = [fR_af^{-1},fR_bf^{-1}] = f[R_a,R_b]f^{-1} = ff^{-1}=1$ shows that $f(a)\sim f(b)$.
\end{proof}

\begin{example}
The relation $\sim$ is not necessarily transitive, even when $Q$ is a connected quandle. Let $Q$ be \emph{\texttt{ConnectedQuandle(10,1)}}, that is, the connected quandle with multiplication table
\begin{displaymath}
    \begin{array}{r|rrrrrrrrrr}
        &0&1&2&3&4&5&6&7&8&9\\
        \hline
        0&0&0&5&9&6&2&4&0&0&3\\
        1&1&1&6&1&5&4&2&8&7&1\\
        2&5&6&2&2&2&0&1&9&2&7\\
        3&9&3&3&3&8&7&3&5&4&0\\
        4&6&5&4&8&4&1&0&4&3&4\\
        5&2&4&0&7&1&5&5&3&5&5\\
        6&4&2&1&6&0&6&6&6&9&8\\
        7&7&8&9&5&7&3&7&7&1&2\\
        8&8&7&8&4&3&8&9&1&8&6\\
        9&3&9&7&0&9&9&8&2&6&9
    \end{array}
\end{displaymath}
Then $R_0=(2,5)(3,9)(4,6)$, $R_1=(2,6)(4,5)(7,8)$, $R_3=(0,9)(4,8)(5,7)$, $[R_0,R_1]=[R_1,R_3]=1$ but $[R_0,R_3]\ne 1$.
\end{example}

\begin{example}
This example shows that not every $R$-clique of a quandle is necessarily a subquandle. Consider the quandle with multiplication table
\begin{displaymath}
    \begin{array}{c|ccc}
        &0&1&2\\
        \hline
        0&0&2&0\\
        1&1&1&1\\
        2&2&0&2
    \end{array}\ .
\end{displaymath}
We have $R_0=1$, so $C=\{0,1\}$ is an $R$-clique. But $0\cdot 1=2\not\in C$. 
\end{example}

\begin{lemma}\label{Lm:CloseMlt}
Let $A$ be a subset of a rack $Q$. Then $\langle R_a:a\in A\rangle = \langle R_x:x\in\langle A\rangle\rangle$.
\end{lemma}
\begin{proof}
Obviously, $G=\langle R_a:a\in A\rangle \le \langle R_x:x\in\langle A\rangle\rangle$. For the other inclusion, consider $x\in\langle A\rangle$. Then $x$ is a word in the generators $A$. We will prove by induction on the length of $x$ that $R_x\in G$. If $x=a\in A$, we are done. Suppose that $x=uv$ or $x=u\rdiv v$ and assume for the induction step that $R_u,R_v\in G$. If $x=uv$ then $R_x = R_{uv} = R_vR_uR_v^{-1}\in G$ and if $x=u\rdiv v$ then $R_x = R_{u\rdiv v} = R_v^{-1}R_uR_v\in G$.
\end{proof}

\begin{lemma}\label{Lm:CloseR}
Let $Q$ be a rack and $C$ an $R$-clique of $Q$. Then $\langle C\rangle$ is an $R$-clique of $Q$.
\end{lemma}
\begin{proof}
Since $C$ is an $R$-clique, the group $\langle R_a:a\in C\rangle$ is abelian. By Lemma \ref{Lm:CloseMlt}, $\langle R_a:a\in C\rangle = \langle R_x:x\in\langle C\rangle \rangle$, which shows that $\langle C\rangle$ is an $R$-clique.
\end{proof}

\subsection{Maximal R-cliques}

Let us have a look at maximal $R$-cliques, particularly in connected quandles.

The \emph{Cayley kernel relation} $\equiv$ on a rack $Q$ is defined by
\begin{displaymath}
    a\equiv b\text{ if and only if }R_a=R_b.
\end{displaymath}
It is easy to see that $\equiv$ is an equivalence relation. In fact, it is a congruence on $Q$. Indeed, if $x\equiv y$ and $u\equiv v$ then $R_{xu} = R_uR_xR_u\inv = R_vR_yR_v\inv = R_{yv}$ and hence $xu\equiv yv$. Similarly, $(x/u)\equiv (y/v)$.

\begin{prop}\label{Pr:Subquandle}
Let $Q$ be a rack and let $C$ be a maximal $R$-clique of $Q$. Then $C$ is a subrack of $Q$. Moreover, $C$ is a union of some congruence classes of the Cayley kernel relation.
\end{prop}
\begin{proof}
By Lemma \ref{Lm:CloseR}, $\langle C\rangle$ is also an $R$-clique. Since $C\le\langle C\rangle$ and $C$ is a maximal $R$-clique, we must have $C=\langle C\rangle$. If $a\in C$ and $a\equiv b$ then $R_a=R_b$ and hence $R_b$ commutes with the same right translations as $R_a$. Since $C$ is maximal, $b\in C$. It follows that $C$ is a union of some equivalence classes of $\equiv$.
\end{proof}

A maximal $R$-clique of a connected quandle might be a union of several equivalence classes of $\equiv$, as can be seen in Example \ref{Ex:6}.

\begin{example}
Maximal $R$-cliques of a quandle need not all have the same size, nor does the size of a maximal $R$-clique have to divide the size of the quandle. In the quandle with multiplication table
\begin{displaymath}
    \begin{array}{c|ccccc}
       &0&1&2&3&4\\
       \hline
       0&0&0&1&1&1\\
       1&1&1&0&0&0\\
       2&3&4&2&4&3\\
       3&4&2&4&3&2\\
       4&2&3&3&2&4
    \end{array}
\end{displaymath}
the subsets $\{0,1\}$ and $\{2\}$ are both maximal $R$-cliques. 
\end{example}

\begin{example}
Maximal $R$-cliques of a given quandle can intersect nontrivially. In the quandle 
\begin{displaymath}
    \begin{array}{c|cccc}
        &0&1&2&3\\
        \hline
        0&0&0&0&0\\
        1&1&1&3&2\\
        2&2&3&2&1\\
        3&3&2&1&3  
    \end{array}
\end{displaymath}
the maximal $R$-cliques are $\{0,1\}$, $\{0,2\}$ and $\{0,3\}$.
\end{example}

\begin{corollary}
Let $Q>1$ be a connected quandle and let $C$ be an $R$-clique of $Q$. Then $C<Q$.
\end{corollary}
\begin{proof}
If $C=Q$ then $\rmlt{Q}$ is abelian and $|Q|=1$ by Proposition \ref{Pr:ConnAndAb}.
\end{proof}

\begin{corollary}
Let $Q$ be a quandle and let $C$ be a maximal $R$-clique of $Q$. If $C>1$ then the subquandle $C$ is not connected.
\end{corollary}
\begin{proof}
By Proposition \ref{Pr:Subquandle}, $C$ is a subquandle of $Q$. Suppose that $C$ is connected. Since $\langle R_a:a\in C\rangle\le\rmlt{Q}$ is abelian, so is $\rmlt{C}$. By Proposition \ref{Pr:ConnAndAb}, $C=1$.
\end{proof}

\subsection{Large R-cliques in connected quandles}

In this subsection we will construct an infinite family of finite connected quandles $Q$ that contain an $R$-clique $C$ such that $|C|=|Q|/3$ and, moreover, the right translations $R_c$ of $Q$ with $c\in C$ are pairwise distinct. This will show that for any multi-virtual link diagram there is a connected quandle such that all virtual crossing types in the diagram can be assigned right translations that pairwise commute and are pairwise distinct.

\def\Z{\mathbb Z}

For $m\ge 1$ and $e\in\Z_2^m$ define $Q_m(e) = (\Z_3\times\Z_2^m,*)$ by
\begin{displaymath}
    (i,a)*(j,b) = \left\{\begin{array}{ll}
        (-i-j,a),&\text{if $i-j\equiv 0\pmod 3$},\\
        (-i-j,a+b),&\text{if $i-j\equiv 1\pmod 3$},\\
        (-i-j,a+b+e),&\text{if $i-j\equiv 2\pmod 3$}.
    \end{array}\right.
\end{displaymath}

\begin{prop}\label{Pr:Construction}
Let $m\ge 1$ and $e\in\Z_2^m$. Then $Q_m(e)$ is a connected quandle of size $3\cdot 2^m$ in which the $2^m$ right translations $R_{(0,a)}$ with $a\in\Z_2^m$ are pairwise distinct and pairwise commute.
\end{prop}
\begin{proof}
Idempotence: We have $(i,a)*(i,a) = (-2i,a) = (i,a)$ since $3i=0$.

Right quasigroup: We will show that $R_{(i,a)}$ is injective. Suppose that $(j,b)*(i,a)=(k,c)*(i,a)$. Comparing the first coordinates, we get $j=k$ and hence $(j,b)*(i,a)=(j,c)*(i,a)$, forcing $b=c$.

Right distributivity: We need to verify the identity
\begin{equation}\label{Eq:RDx}
    ((i,a)*(k,c))*((j,b)*(k,c)) = ((i,a)*(j,b))*(k,c).
\end{equation}
Let us write $(i,a)*_{i-j}(j,b)$ instead of $(i,a)*(j,b)$ to better indicate which of the three cases of the definition of $*$ applies. Suppose that $i-k\equiv r$ and $j-k\equiv s$. Then $i-j \equiv (i-k)-(j-k) \equiv r-s$, $(-i-j)-k\equiv -(i-k)-(j-k)\equiv -r-s$ and $(-i-k)-(-j-k)\equiv -i+j \equiv s-r$. We can therefore rewrite \eqref{Eq:RDx} in more detail as
\begin{equation}\label{Eq:RDy}
    ((i,a)*_r(k,c))*_{s-r}((j,b)*_s(k,c)) = ((i,a)*_{r-s}(j,b))*_{-r-s}(k,c).
\end{equation}
In the first coordinate, we obtain $-(-i-k)-(-j-k)=i+j+2k$ on the left hand side and $-(-i-j)-k = i+j-k$ on the right hand size. The results coincide thanks to $3k\equiv 0$. Since the first coordinate is now taken care of, let us focus on the second coordinate and let us rewrite \eqref{Eq:RDy} with a slight abuse of notation as
\begin{equation}\label{Eq:RDz}
    (a*_rc)*_{s-r}(b*_s c) = (a*_{r-s} b)*_{-r-s}*c.
\end{equation}
We will verify \eqref{Eq:RDz} for all nine choices of $(r,s)$. We obtain the following left hand and right hand sides of \eqref{Eq:RDz}:
\begin{displaymath}
    \begin{array}{ll}
        r=0,\,s=0:\quad a*_0b = a,                        &a*_0c = a,\\
        r=0,\,s=1:\quad a*_1(b+c) = a+b+c,                &(a+b+e)*_2 c = a+b+c,\\
        r=0,\,s=2:\quad a*_2(b+c+e) = a+b+c,              &(a+b)*_1c = a+b+c,\\
        r=1,\,s=0:\quad (a+c)*_2 b = a+b+c+e,             &(a+b)*_2c = a+b+c+e,\\
        r=1,\,s=1:\quad (a+c)*_0(b+c) = a+c,              &a*_1c = a+c,\\
        r=1,\,s=2:\quad (a+c)*_1(b+c+e) = a+b+e,          &(a+b+e)*_0c = a+b+e,\\
        r=2,\,s=0:\quad (a+c+e)*_1b = a+b+c+e,            &(a+b+e)*_1c = a+b+c+e,\\
        r=2,\,s=1:\quad (a+c+e)*_2(b+c) = a+b,            &(a+b)*_0c = a+b,\\
        r=2,\,s=2:\quad (a+c+e)*_0(b+c+e) = a+c+e,        &a*_2c = a+c+e. 
    \end{array}
\end{displaymath}
Hence $Q_m(e)$ is a quandle.

Connectedness: It suffices to show that the orbit of $(0,0)$ under the action of $\rmlt{Q_m(e)}$ is all of $Q_m(e)$. Let $a\in\Z_2^m$ be given. Then $(0,0)*(2,a)=(1,a)$, $((0,0)*(2,a))*(2,e) = (1,a)*(2,e) = (0,a)$ and $(0,0)*(1,a+e) = (2,a)$.

It remains to show that the $2^m$ right translations $R_{(0,a)}$ with $a\in\Z_2^m$ are pairwise distinct and pairwise commuting. If $a\ne b$ then $R_{(0,a)}(1,0) = (1,0)*(0,a) = (2,a)\ne (2,b) = R_{(0,b)}(1,0)$ and so $R_{(0,a)}\ne R_{(0,b)}$. We have
\begin{displaymath}
    R_{(0,a)}R_{(0,b)}(k,c) = \left\{\begin{array}{ll}
        (0,c)*(0,a)=(0,c),&\text{if $k=0$},\\
        (-1,c+b)*(0,a) = (1,a+b+c+e),&\text{if $k=1$},\\
        (-2,c+b+e)*(0,a) = (2,a+b+c+e),&\text{if $k=2$}.
    \end{array}\right.    
\end{displaymath}
Since in each case the answer is invariant under the interchange of $a$ and $b$, we conclude that $R_{(0,a)}R_{(0,b)} = R_{(0,b)}R_{(0,a)}$.
\end{proof}

\begin{example}\label{Ex:6}
For $m=1$ and $e=0$ the construction of Proposition \ref{Pr:Construction} yields the connected quandle
\begin{displaymath}
    \begin{array}{c|cccccc}
         Q_1(0) &(0,0)&(0,1)&(1,0)&(1,1)&(2,0)&(2,1)\\
                \hline
         (0,0)  &(0,0)&(0,0)&(2,0)&(2,1)&(1,0)&(1,1)\\
         (0,1)  &(0,1)&(0,1)&(2,1)&(2,0)&(1,1)&(1,0)\\
         (1,0)  &(2,0)&(2,1)&(1,0)&(1,0)&(0,0)&(0,1)\\
         (1,1)  &(2,1)&(2,0)&(1,1)&(1,1)&(0,1)&(0,0)\\
         (2,0)  &(1,0)&(1,1)&(0,0)&(0,1)&(2,0)&(2,0)\\
         (2,1)  &(1,1)&(1,0)&(0,1)&(0,0)&(2,1)&(2,1)
         
    \end{array}\ .
\end{displaymath}
It turns out that $Q_1(0)$ is isomorphic to \texttt{ConnectedQuandle(6,1)}. One can also check that $Q_1(1)$ is isomorphic to \texttt{ConnectedQuandle(6,2)} and $Q_2(0,0)$ is isomorphic to \texttt{ConnectedQuandle(12,8)}, for instance.
\end{example}

We conclude our algebraic exploration of operator quandles with several open problems:

\begin{problem}
Do all maximal $R$-cliques in a finite connected rack have the same size?
\end{problem}

\begin{problem}
Do maximal $R$-cliques in a finite connected rack $Q$ partition $Q$?
\end{problem}

\begin{problem}
Does there exist a finite connected rack $Q$ and a maximal $R$-clique $C$ of $Q$ such that $|C|$ does not divide $|Q|$?
\end{problem}

\begin{problem}
Does there exist a finite connected rack $Q>1$ and a maximal $R$-clique $C$ of $Q$ such that $|C|>|Q|/3$ (or even $|C|>|Q|/2$)? 
\end{problem}

\section{Distinguishing small oriented multi-virtual knots}\label{Sc:Distinguish}

Let $\textbf V=\{2_1,3_1,3_2,\dots,3_7\}$ be the set of diagrams of all virtual knots with at most three classical crossings as presented in the catalog \cite{Gre}.

Let $\textbf M$ be the set of all oriented multi-virtual knot diagrams with a virtual projection in $\textbf V$. Diagrams of the knots in $\textbf M$ can be found in the Appendix. For each such diagram we picked one of the two possible orientations arbitrarily and we labeled virtual crossings by generic types. Note that each diagram in the Appendix represents many multi-virtual knot diagrams, depending on the actual values of the types assigned to virtual crossings. 

\begin{rem}
We do not claim that $\textbf M$ contains all multi-virtual knots with at most $3$ classical crossings. For instance, the multi-virtual trefoil knot from Figure \ref{Fg:Trefoils} (which we suspect is not an unknot but we do not know it for sure), is not present in $\textbf M$ since its virtual projection is an unknot. Furthermore, below we will construct an infinite family of pairwise nonequivalent multi-virtual knots, each with a single classical crossing, and these knots are not found in $\textbf M$.
\end{rem}

We proceed to sort the diagrams of $\textbf M$ up to equivalence using the two invariants developed in Section \ref{Sc:Invariants}. Since equivalent multi-virtual knot diagrams have equivalent virtual projections, it suffices to take one diagram $V\in\textbf V$ at a time and consider the set $m(V)\subseteq\textbf M$.

The knot $3_6$ is the classical trefoil knot and there is nothing to do.

The knot $2_1$ contains precisely one virtual crossing, and our task is therefore to decide if the two versions of $2_1$ with virtual crossings of different types are equivalent. Similarly for the knot $3_2$.

The remaining knots $3_1$, $3_3$, $3_4$, $3_5$ and $3_7$ of $\textbf{V}$ contain precisely two virtual crossings. As explained in Subsection \ref{Ss:SameVirtual}, for each of these knots we must in principle solve the $B_4=15$ equivalence problems obtained by assigning types to virtual crossings as in \eqref{Eq:2Plus2}. We will be a bit more efficient.

\subsection{Operator quandle colorings of small multi-virtual knots}

Let $D$ be one of the diagrams of $\textbf M$ and let $(Q,*,A_T)$ be an operator quandle with right division operation $/$. In order to count all operator quandle colorings of $D$ by $(Q,*,A_T)$, we label one or two arcs of the diagram by $a$ and $b$, and then chase the coloring rule along the diagram until all arcs of $D$ are labeled. Whenever we return to an arc labeled already, we record the closure condition that must be satisfied at that point so that the labeling gives rise to an operator quandle coloring. We indicate the location of the closure check points by the symbols $\color{red}\bullet$ or $\color{red}\circ$ in the diagram.

For instance, looking at the diagram for $2_1$ in the Appendix, we see that the only condition that arises for the knot $2_1$ is $(a/\alpha(a))/\alpha^{-1}(a) = a$, which can be rewritten as 
\begin{equation}\label{Eq:21Condition}
    \alpha(a/\alpha(a)) = \alpha(a)*a    
\end{equation}
upon multiplying on the right by $\alpha^{-1}(a)$ and applying $\alpha$ to both sides. Hence $\mathrm{Col}(2_1,Q,*,(\alpha))$ is the cardinality of the set $\{a\in Q:\alpha(\alpha(a)/a) = \alpha(a)*a\}$.

We have written a \texttt{GAP} code that counts the number of operator quandle colorings once the operator quandle and the closure conditions analogous to \eqref{Eq:21Condition} are given. The code can be downloaded from the homepage of the third author. In what follows we will present a few calculations by hand. All other results were obtained by computer.

We will get a lot of mileage from the following quandle of size $4$. Let $R=\Z_2[t^{\pm 1}]/(t^2+t+1)$. Since $t^2+t+1\equiv 0$, we have $0=t^{-1}0 \equiv t^{-1}(t^2+t+1) \equiv t+1+t^{-1}$ and hence $t^{-1}\equiv 1+t$. By identifying the cosets of $R$ with their representatives, we can therefore list the four elements of $R$ as $0,1,t,t^{-1}$. The addition and multiplication in $R$ is then given by
\begin{displaymath}
    \begin{array}{c|cccc}
        (R,+)&0&1&t&t^{-1}\\
        \hline
        0&0&1&t&t^{-1}\\
        1&1&0&t^{-1}&t\\
        t&t&t^{-1}&0&1\\
        t^{-1}&t^{-1}&t&1&0
    \end{array}
    \qquad\qquad
    \begin{array}{c|cccc}
        (R,\cdot)&0&1&t&t^{-1}\\
        \hline
        0&0&0&0&0\\
        1&0&1&t&t^{-1}\\
        t&0&t&t^{-1}&1\\
        t^{-1}&0&t^{-1}&1&t\\
    \end{array}
\end{displaymath}

Let $Q_4 = (R,*)$ be the Alexander quandle on $R$ given, as usual, by $a*b = ta + (1-t)b$. The multiplication $*$ and the right division $/$ in $Q_4$ are then given by
\begin{displaymath}
    \begin{array}{c|cccc}
        (Q_4,*)&0&1&t&t^{-1}\\
        \hline
        0&0&t^{-1}&1&t\\
        1&t&1&t^{-1}&0\\
        t&t^{-1}&0&t&1\\
        t^{-1}&1&t&0&t^{-1}.
    \end{array}
    \qquad\qquad
    \begin{array}{c|cccc}
        (Q_4,/)&0&1&t&t^{-1}\\
        \hline
        0&0&t&t^{-1}&1\\
        1&t^{-1}&1&0&t\\
        t&1&t^{-1}&t&0\\
        t^{-1}&t&0&1&t^{-1}.
    \end{array}
\end{displaymath}

The automorphism group $Q_4$ is isomorphic to the alternating group $A_4$. It can be checked that the $3$-cycle
\begin{equation} 
    \theta = (1,t,t^{-1})
\end{equation}
is an automorphism of $Q_4$.

\begin{example}
We will show that if the type of the unique virtual crossing of $2_1$ is changed, the two resulting multi-virtual knots are not equivalent. It suffices to show that $\mathrm{Col}(2_1,Q_4,*,(1))\ne\mathrm{Col}(2_1,Q_4,*,(\theta^{-1}))$, for instance. With $\alpha=1$, the compatibility condition \eqref{Eq:21Condition} becomes $a/a = a*a$, which holds for every $a\in Q$ by idempotence, and hence $\mathrm{Col}(2_1,Q_4,*,(1))=4$. With $\alpha=\theta^{-1} = (1,t^{-1},t)$, the two sides of the condition \eqref{Eq:21Condition} evaluate as follows:
\begin{displaymath}
    \begin{array}{lll}
        a=0:        &\alpha(0/\alpha(0)) = \alpha(0/0)=\alpha(0)=0,                         &\alpha(0)*0=0*0=0,\\
        a=1:        &\alpha(1/\alpha(1)) = \alpha(1/t^{-1})=\alpha(t)=1,                    &\alpha(1)*1 = t^{-1}*1=t,\\
        a=t:        &\alpha(t/\alpha(t)) = \alpha(t/1) = \alpha(t^{-1})=t,                  &\alpha(t)*t = 1*t = t^{-1},\\
        a=t^{-1}:   &\alpha(t^{-1}/\alpha(t^{-1})) = \alpha(t^{-1}/t) = \alpha(1) = t^{-1}, &\alpha(t^{-1})*t^{-1} = t*t^{-1}=1.   
    \end{array}
\end{displaymath}
Hence $\mathrm{Col}(2_1,Q_4,*,(\theta^{-1}))=1$.
\end{example}

Similarly, $\mathrm{Col}(3_2,Q_3,*,(1))=4\ne 1=\mathrm{Col}(3_2,Q_4,*,(\theta^{-1}))$ shows that changing the type of the unique virtual crossing of $3_2$ yields a nonequivalent multi-virtual knot.

For diagrams $D$ with two virtual crossings, we proceed systematically as follows. According to \eqref{Eq:2Plus2}, we need to consider 15 situations and up to four types of virtual crossings at a time. Note that every short column (the two top rows or the two bottom rows) of \eqref{Eq:2Plus2} is recorded in
\begin{displaymath}
    U=\{(1,1),(1,2),(2,1),(2,2),(2,3),(1,3),(3,1),(3,2),(3,3),(3,4)\}.
\end{displaymath}
Given a quandle $(Q,*)$, let $A$ be the set of all quadruples of distinct, pairwise commuting automorphisms of $\mathrm{Aut}(Q)$. For each $u=(u_1,u_2)\in U$ we consider all quadruples $(\alpha_1,\alpha_2,\alpha_3,\alpha_4)\in A$, assign the automorphisms $(\alpha_{u_1},\alpha_{u_2})$ to the two types in $D$, and calculate the resulting operator quandle invariant. For each $u\in U$ we therefore obtain a vector $q(u)$ of invariants of length $|A|$. If for some $u,v\in U$ we have $q(u)\ne q(v)$, the two corresponding multi-virtual diagrams are not equivalent.

Using this systematic strategy, computer calculations with the quandle $Q_4$ show that all $|U|=10$ multi-virtual versions of each of the knots $3_1$, $3_3$ and $3_4$ are pairwise nonequivalent. However, for each of the knots $3_5$ and $3_7$ the computation is not able to distinguish between the two multi-virtual versions with the two types interchanged. For instance, the knot corresponding to $(1,2)$ gives the same invariant as the knot corresponding to $(2,1)$. We settle these two cases with a different quandle:

\begin{example}\label{Ex:35Calc}
Let $Q$ be \texttt{ConnectedQuandle(5,2)}, that is, the quandle with multiplication and right division given by
\begin{displaymath}
    \begin{array}{c|ccccc}
        *&1&2&3&4&5\\
        \hline
        1&1&5&4&3&2\\
        2&3&2&1&5&4\\
        3&5&4&3&2&1\\
        4&2&1&5&4&3\\
        5&4&3&2&1&5
    \end{array}
    \qquad
    \begin{array}{c|ccccc}
        /&1&2&3&4&5\\
        \hline
        1&1&4&2&5&3\\
        2&4&2&5&3&1\\
        3&2&5&3&1&4\\
        4&5&3&1&4&2\\
        5&3&1&4&2&5
    \end{array}
\end{displaymath}
The automorphism group of $Q$, calculated with the package \texttt{RightQuasigroups}, contains the automorphism $\omega=(2,4,5,3)$. For $3_5$ we need to verify the condition $((a/\alpha\beta(a))/\alpha\beta^{-1}(a))/\alpha^{-1}\beta^{-1}(a)=a$. Multiplying on the right by $\alpha^{-1}\beta^{-1}(a)$ and then substituting $\beta\alpha(a)$ for $a$ yields the condition $(\alpha\beta(a)/(\alpha\beta)^2(a))/\alpha^2(a) = \beta\alpha(a)*a$, which we can also write as
\begin{equation}\label{Eq:35Condition1}
    \beta\alpha(a)*a = \alpha(\beta( a/\beta\alpha(a))/\alpha(a)).
\end{equation}
With $(\alpha,\beta)=(1,\omega)$, this reduces to
\begin{equation}\label{Eq:35Condition2}
    \omega(a)*a = \omega(a/\omega(a))/a.
\end{equation}
Comparing the two sides of \eqref{Eq:35Condition2} yields
\begin{displaymath}
    \begin{array}{lll}
        a=1: &\omega(1)*1 = 1*1 = 1,    &\omega(1/\omega(1))/1 = \omega(1/1)/1 = \omega(1)/1 = 1/1 = 1,\\
        a=2: &\omega(2)*2 = 4*2 = 1,    &\omega(2/\omega(2))/2 = \omega(2/4)/2 = \omega(3)/2 = 2/2 = 2,\\
        a=3: &\omega(3)*3 = 2*3 = 1,    &\omega(3/\omega(3))/3 = \omega(3/2)/3 = \omega(5)/3 = 3/3 = 3,\\
        a=4: &\omega(4)*4 = 5*4 = 1,    &\omega(4/\omega(4))/4 = \omega(4/5)/4 = \omega(2)/4 = 4/4 = 4,\\
        a=5: &\omega(5)*5 = 3*5 = 1,    &\omega(5/\omega(5))/5 = \omega(5/3)/5 = \omega(4)/5 = 5/5 = 5.\\
    \end{array}
\end{displaymath}
Hence $\mathrm{Col}(3_5,Q,*,(1,\omega)) = 1$. With $(\alpha,\beta)=(\omega,1)$, the condition \eqref{Eq:35Condition1} reduces to
\begin{equation}\label{Eq:35Condition3}
    \omega(a)*a = \omega((a/\omega(a))/\omega(a)).
\end{equation}
The left hand side of \eqref{Eq:35Condition3} is the same as the left hand side of \eqref{Eq:35Condition2}. The right hand side of \eqref{Eq:35Condition3} yields
\begin{displaymath}
    \begin{array}{ll}
        a=1: &\omega((1/\omega(1))/\omega(1)) = \omega((1/1)/1) = \omega(1/1) = \omega(1) = 1,\\
        a=2: &\omega((2/\omega(2))/\omega(2)) = \omega((2/4)/4) = \omega(3/4) = \omega(1) = 1,\\
        a=3: &\omega((3/\omega(3))/\omega(3)) = \omega((3/2)/2) = \omega(5/2) = \omega(1) = 1,\\
        a=4: &\omega((4/\omega(4))/\omega(4)) = \omega((4/5)/5) = \omega(2/5) = \omega(1) = 1,\\
        a=5: &\omega((5/\omega(5))/\omega(5)) = \omega((5/3)/3) = \omega(4/3) = \omega(1) = 1.    
    \end{array}
\end{displaymath}
Hence $\mathrm{Col}(3_5,Q,*,(\omega,1)) = 5$. Interchanging the two distinct types of virtual crossings in $3_5$ therefore yields nonequivalent multi-virtual knots.
\end{example}

Similarly, one can check that with the same choice of a quandle and automorphisms as in Example \ref{Ex:35Calc}, the two versions of $3_7$ in question yield nonequivalent multi-virtual knots.

\medskip

We have now proved that with the choice of orientation as in the Appendix, all versions of ``small'' multi-virtual knots are pairwise nonequivalent.

\medskip

Let $K$ and $K'$ be the two permutationally equivalent diagrams in Figure \ref{Fg:LinksPair}. It appears that the multi-virtual arrow polynomials for $K$ and $K'$ are distinct (see \cite{Kau2} for the definition of the multi-virtual arrow polynomial), however, we do not yet have a proof. Instead, we will show below that $K$ and $K'$ are distinguished by the operator quandle $2$-cocycle invariant.

\begin{example}\label{Ex:Special}
Let $(Q,*)$ be the Alexander quandle over $R=\Z[t^{\pm 1}]/(t^4+t^3+t^2+t+1)$. It is a quandle of size $16$ with an automorphism group of size $240$ which contains the commuting automorphisms
\begin{displaymath}
    \sigma = ( t^3, t^2+t^3, t+t^3, 1+t+t^2+t^3, 1+t+t^2, 1+t+t^3, 1, t+t^2+t^3, 1+t^3, t,  1+t, 1+t^2+t^3, t^2, t+t^2, 1+t^2 )
\end{displaymath}
and
\begin{displaymath}
    \tau = ( t^3, t, 1+t+t^2+t^3, t^2, 1 )(t^2+t^3, 1+t, 1+t+t^2, t+t^2, t+t^2+t^3 )( t+t^3, 1+t^2+t^3, 1+t+t^3, 1+t^2, 1+t^3 ).
\end{displaymath}    
Let $D$ be the left diagram in Figure \ref{Fg:LinksPair}. We claim that $\mathrm{Col}(D,Q,*,(\sigma,\tau))\ne\mathrm{Col}(D,Q,*,(\tau,\sigma))$.

Figure \ref{Fg:Special} shows that once the two arcs of $D$ labeled by $a,b\in Q$ are given, the labeling of the remaining arcs is forced. The labeling gives rise to an operator quandle coloring if and only if the conditions
\begin{displaymath}
    \alpha\beta(a)/\alpha\beta^{-2}(b)=a\quad\text{and}\quad\alpha^{-1}\beta^{-1}(b)*\alpha(a)=b
\end{displaymath}
hold. These two conditions can be routinely checked by a computer, giving $\mathrm{Col}(D,Q,*,(\sigma,\tau))=16$ and $\mathrm{Col}(D,Q,*,(\tau,\sigma))=1$. Therefore, if $\alpha\ne\beta$, the two multi-virtual links in Figure \ref{Fg:LinksPair} are not equivalent.
\end{example}

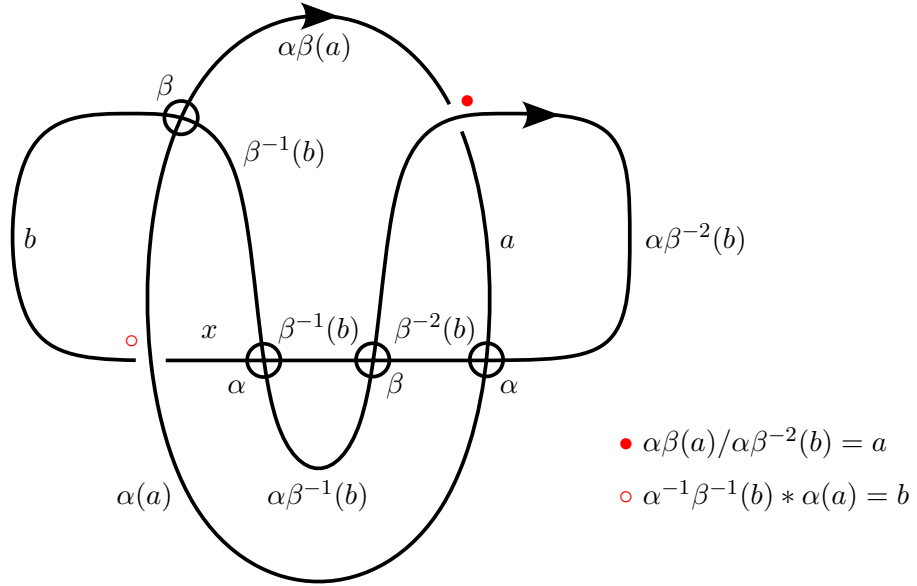
\begin{figure}[!ht]
\centering
\begin{overpic}[width = 0.5\textwidth]{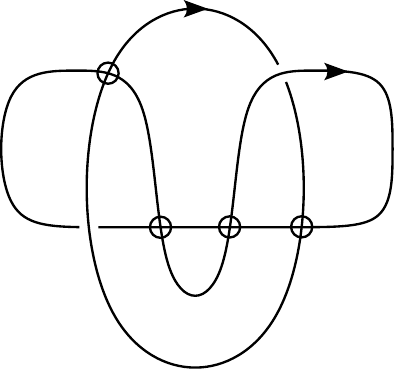}
\put(55,185){$\beta$}
\put(82,72){$\alpha$}
\put(142,72){$\beta$}
\put(185,72){$\alpha$}
\put(185,127){$a$}
\put(40,30){$\alpha(a)$}
\put(101,200){$\alpha\beta(a)$}
\put(5,127){$b$}
\put(88,160){$\beta^{-1}(b)$}
\put(97,30){$\alpha\beta^{-1}(b)$}
\put(240,127){$\alpha\beta^{-2}(b)$}
\put(145,93){$\beta^{-2}(b)$}
\put(101,93){$\beta^{-1}(b)$}
\put(72,93){$x$}
\put(170,180){$\color{red}\bullet$}
\put(43,89){$\color{red}\circ$}
\put(230,50){$\color{red}\bullet$\ $\alpha\beta(a)/\alpha\beta^{-2}(b)=a$}
\put(230,30){$\color{red}\circ$\ $\alpha^{-1}\beta^{-1}(b)*\alpha(a)=b$}
\end{overpic}
\caption{Operator quandle coloring conditions for a multi-virtual link. Here, $x=\alpha^{-1}\beta^{-1}(b)$.}\label{Fg:Special}
\end{figure}

\subsection{Cocycle invariants of small multi-virtual knots}

To illustrate the cocycle invariant for multi-virtual knots, let us use it to once again distinguish the two multi-virtual versions of $3_7$ obtained by interchanging the two types.

Let $Q_4=(R,*)$ be the Alexander quandle of order $4$ as above and let $A=\mathrm{Aut}(Q_4)$. It has been shown in \cite{Ame} that $\phi:Q_4\times Q_4\to (R,+)$ given by
\begin{displaymath}
    \phi(x,y) = y(x-y)^2
\end{displaymath}
is a cocycle of $Q_4$. (Here the operations $\cdot$, $+$ and $-$ are taken in $R=\Z_2[t^{\pm 1}]/(t^2+t+1)$.) For the convenience of the reader, here is the table of $\phi$:
\begin{displaymath}
    \begin{array}{c|cccc}
        \phi&0&1&t&t^{-1}\\
        \hline
        0&0&1&1&1\\
        1&0&0&t^{-1}&t\\
        t&0&t&0&t^{-1}\\
        t^{-1}&0&t^{1}&t&0
    \end{array}
\end{displaymath}
It turns out that the set
\begin{displaymath}
    \{g\in A: \phi(a,b) = \phi(g(a),g(b))\text{ for all }a,b\in Q_4\}
\end{displaymath}
of all automorphisms of $Q_4$ compatible with $\phi$ is precisely the subgroup $\langle\theta\rangle\le A$, where, as above, $\theta=(1,t,t^{-1})$. Since $\langle\theta\rangle$ is cyclic, any two automorphisms of $\langle\theta\rangle$ commute.

Consider the operator quandle $(Q_4,*,(\alpha,\beta))$ for some commuting automorphisms $\alpha$ and $\beta$ of $Q_4$. Let us assign the colors $a,b\in Q_4$ to the two selected arcs of $3_7$ as in the Appendix. This gives rise to an operator quandle coloring if and only if the conditions
\begin{equation}\label{Eq:37conditions}
    a*b = \beta^{-1}(\alpha^{-1}(b)*\alpha(a))\quad\text{and}\quad \beta\alpha(a) = b*\beta^{-1}(\alpha^{-1}(b)*\alpha(a))
\end{equation}
hold. 

Let us set $(\alpha,\beta) = (\theta,\theta^{-1})$. Then \eqref{Eq:37conditions} becomes
\begin{displaymath}
    a*b = b*\theta^{-1}(a)\quad\text{and}\quad a = b*(b*\theta^{-1}(a)),
\end{displaymath}
where we have used $\theta^3=1$. It can be checked that there are then precisely four operator quandle colorings of $3_7$, namely those given by $(a,b)\in\{(0,0),(t,0),(1,0),(t^{-1},0)\}$. These four colorings $C_1,\dots,C_4$ are depicted in Figure \ref{Fg:37ColoringsA}.

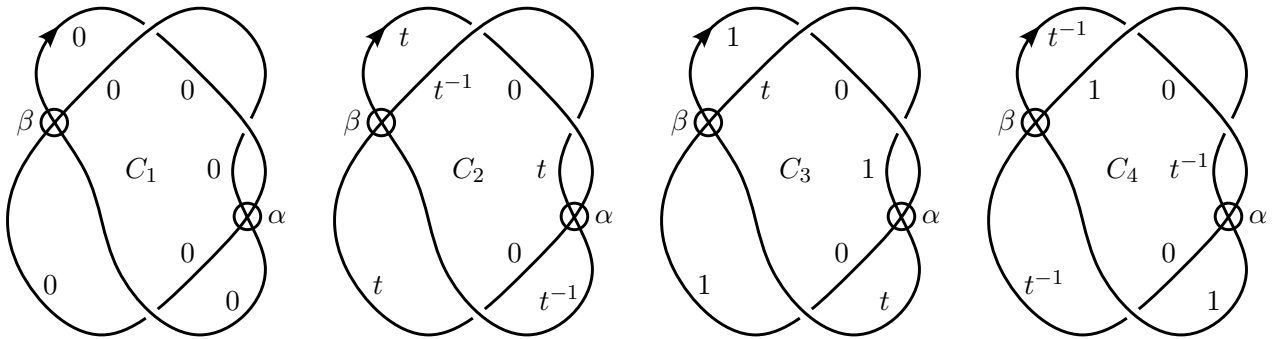
\begin{figure}[!htb]
    \centering
    \begin{minipage}{.21\textwidth}
        \begin{overpic}[width = 1.0\textwidth]{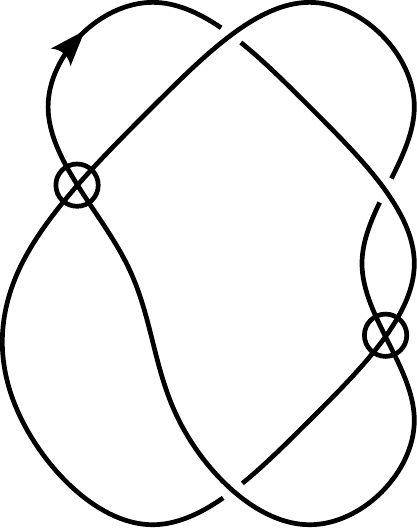}
            \put(99,43){$\alpha$}
            \put(4,78){$\beta$}
            \put(25,110){$0$}
            \put(66,90){$0$}
            \put(66,28){$0$}
            \put(14,16){$0$}
            \put(38,90){$0$}
            \put(76,60){$0$}
            \put(83,10){$0$}
            \put(45,60){$C_1$}
        \end{overpic}
    \end{minipage}
    \qquad
    \begin{minipage}{0.21\textwidth}
        \begin{overpic}[width = 1.0\textwidth]{37.pdf}
            \put(99,43){$\alpha$}
            \put(4,78){$\beta$}
            \put(25,110){$t$}
            \put(66,90){$0$}
            \put(66,28){$0$}
            \put(15,16){$t$}
            \put(38,90){$t^{-1}$}
            \put(77,60){$t$}
            \put(78,10){$t^{-1}$}
            \put(45,60){$C_2$}
        \end{overpic}
    \end{minipage}
    \qquad
    \begin{minipage}{0.21\textwidth}
        \begin{overpic}[width = 1.0\textwidth]{37.pdf}
            \put(99,43){$\alpha$}
            \put(4,78){$\beta$}
            \put(25,110){$1$}
            \put(66,90){$0$}
            \put(66,28){$0$}
            \put(14,16){$1$}
            \put(38,90){$t$}
            \put(76,60){$1$}
            \put(83,10){$t$}
            \put(45,60){$C_3$}
        \end{overpic}
    \end{minipage}
    \qquad
    \begin{minipage}{0.21\textwidth}
        \begin{overpic}[width = 1.0\textwidth]{37.pdf}
            \put(99,43){$\alpha$}
            \put(4,78){$\beta$}
            \put(23,110){$t^{-1}$}
            \put(66,90){$0$}
            \put(66,28){$0$}
            \put(14,16){$t^{-1}$}
            \put(38,90){$1$}
            \put(69,60){$t^{-1}$}
            \put(83,10){$1$}
            \put(45,60){$C_4$}
        \end{overpic}
    \end{minipage}
    \caption{Four operator quandle colorings for $3_7$ with $(\alpha,\beta)=(\theta,\theta^{-1})$.}\label{Fg:37ColoringsA}
\end{figure}

Summing up over the three classical crossings of $3_7$ from the top to the bottom, the cocycle invariant (with the cocycle $\phi$) therefore becomes
\begin{align*}
    &(-\phi(0,0)-\phi(0,0)+\phi(0,0)) + (-\phi(t,t^{-1})-\phi(t^{-1},0)+\phi(t^{-1},0))\\
    &+(-\phi(1,t)-\phi(t,0)+\phi(t,0)) + (-\phi(t^{-1},1)-\phi(1,0)+\phi(1,0))\\
    &= 0-t^{-1}-t^{-1}-t^{-1} = -3t^{-1}.
\end{align*}

Let us now set $(\alpha,\beta) = (\theta^{-1},\theta)$. Then \eqref{Eq:37conditions} becomes
\begin{displaymath}
    a*b = b*\theta(a)\quad\text{and}\quad a = b*(b*\theta(a)).
\end{displaymath}
It can be checked that there are precisely four operator quandle colorings of $3_7$, namely those given by $(a,b)\in\{(0,0),(t,1),(1,t^{-1}),(t^{-1},t)\}$. These four colorings $C'_1,\dots,C'_4$ are depicted in Figure \ref{Fg:37ColoringsB}.

\begin{figure}[!htb]
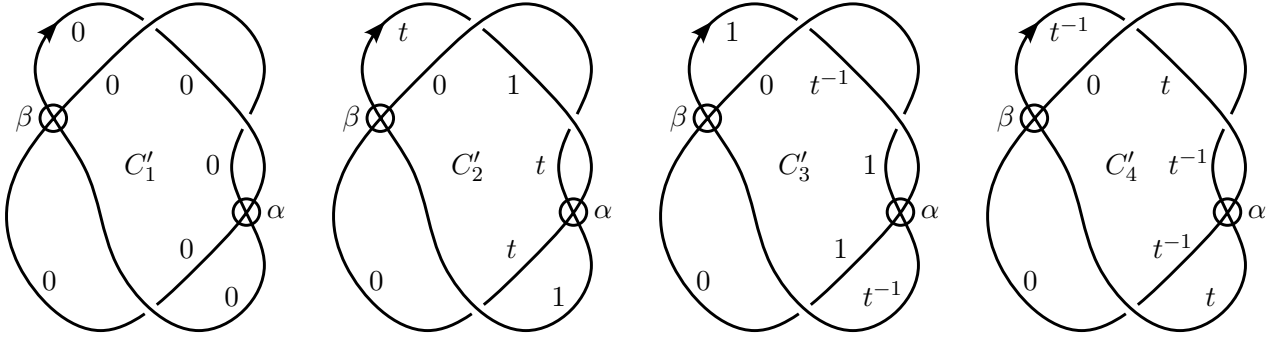

    \centering
    \begin{minipage}{.21\textwidth}
        \begin{overpic}[width = 1.0\textwidth]{37.pdf}
            \put(99,43){$\alpha$}
            \put(4,78){$\beta$}
            \put(25,110){$0$}
            \put(66,90){$0$}
            \put(66,28){$0$}
            \put(14,16){$0$}
            \put(38,90){$0$}
            \put(76,60){$0$}
            \put(83,10){$0$}
            \put(45,60){$C'_1$}
        \end{overpic}
    \end{minipage}
    \qquad
    \begin{minipage}{0.21\textwidth}
        \begin{overpic}[width = 1.0\textwidth]{37.pdf}
            \put(99,43){$\alpha$}
            \put(4,78){$\beta$}
            \put(25,110){$t$}
            \put(66,90){$1$}
            \put(66,28){$t$}
            \put(14,16){$0$}
            \put(38,90){$0$}
            \put(77,60){$t$}
            \put(83,10){$1$}
            \put(45,60){$C'_2$}
        \end{overpic}
    \end{minipage}
    \qquad
    \begin{minipage}{0.21\textwidth}
        \begin{overpic}[width = 1.0\textwidth]{37.pdf}
            \put(99,43){$\alpha$}
            \put(4,78){$\beta$}
            \put(25,110){$1$}
            \put(57,90){$t^{-1}$}
            \put(66,28){$1$}
            \put(14,16){$0$}
            \put(38,90){$0$}
            \put(77,60){$1$}
            \put(77,10){$t^{-1}$}
            \put(45,60){$C'_3$}
        \end{overpic}
    \end{minipage}
    \qquad
    \begin{minipage}{0.21\textwidth}
        \begin{overpic}[width = 1.0\textwidth]{37.pdf}
            \put(99,43){$\alpha$}
            \put(4,78){$\beta$}
            \put(24,110){$t^{-1}$}
            \put(66,90){$t$}
            \put(63,28){$t^{-1}$}
            \put(14,16){$0$}
            \put(38,90){$0$}
            \put(69,60){$t^{-1}$}
            \put(83,10){$t$}
            \put(45,60){$C'_4$}
        \end{overpic}
    \end{minipage}
    \caption{Four operator quandle colorings for $3_7$ with $(\alpha,\beta)=(\theta^{-1},\theta)$.}\label{Fg:37ColoringsB}
\end{figure}

The relevant cocycle invariant becomes
\begin{align*}
    &(-\phi(0,0)-\phi(0,0)+\phi(0,0)) + (-\phi(t,0)-\phi(0,1)+\phi(1,t))\\
    &+(-\phi(1,0)-\phi(0,t^{-1})+\phi(t^{-1},1)) + (-\phi(t^{-1},0)-\phi(0,t)+\phi(t,t^{-1}))\\
    &= 0+(0-1+t^{-1})+(0-1+t^{-1})+(0-1+t^{-1}) = -3+3t^{-1}.
\end{align*}

By Theorem \ref{Th:Coc}, we have now proved once again that interchanging the two distinct types in the multi-virtual knot $3_7$ results in two nonequivalent multi-virtual knots.

\subsection{A family of nonequivalent multi-virtual knots with a single classical crossing}\label{Ss:SingleClassical}

Let $K(n)$ be the oriented multi-virtual knot obtained by a sequence of $2n$ twists with virtual crossings that alternate distinct types $\alpha$ and $\beta$, followed by a single classical crossing, followed by arcs closing the knot. Let the orientation be chosen so that the classical crossing is positive. The knot $K(3)$ is illustrated in Figure \ref{Fg:K3}.

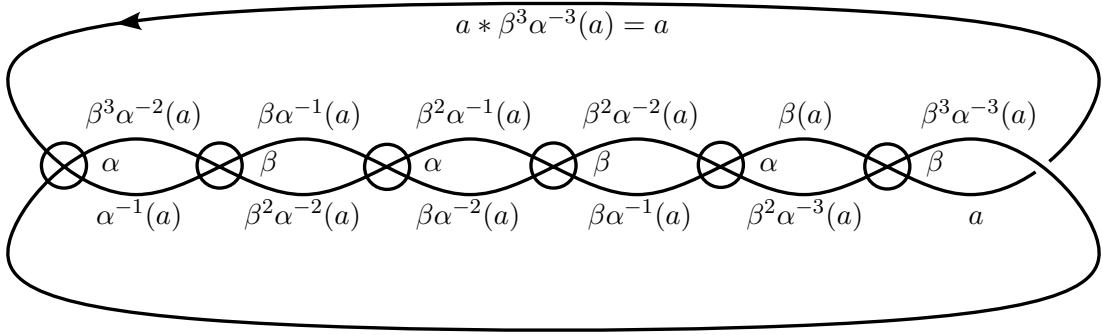
\begin{figure}[!ht]
\centering
\begin{overpic}[width = 0.9\textwidth]{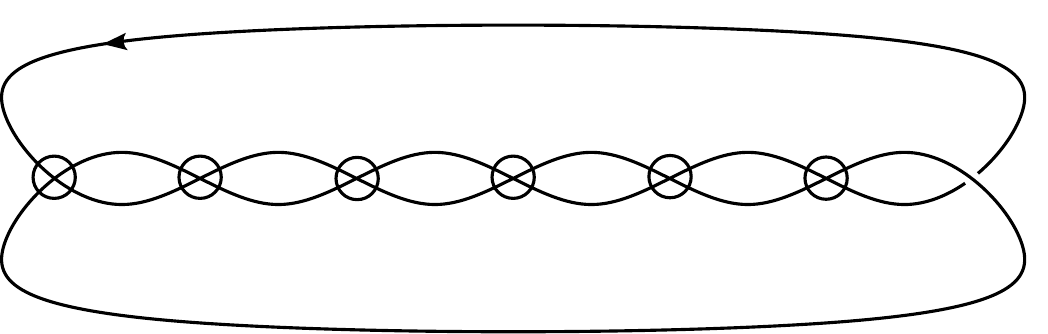}

\put(36,61){$\alpha$}
\put(96,61){$\beta$}
\put(158,61){$\alpha$}
\put(222,61){$\beta$}
\put(285,61){$\alpha$}
\put(348,61){$\beta$}
\put(170,113){$a * \beta^3\alpha^{-3}(a) = a$}
\put(34,41){$\alpha^{-1}(a)$}
\put(30,79){$\beta^3\alpha^{-2}(a)$}
\put(90,41){$\beta^2\alpha^{-2}(a)$}
\put(95,79){$\beta\alpha^{-1}(a)$}
\put(155,41){$\beta\alpha^{-2}(a)$}
\put(155,79){$\beta^2\alpha^{-1}(a)$}
\put(220,41){$\beta\alpha^{-1}(a)$}
\put(218,79){$\beta^2\alpha^{-2}(a)$}
\put(280,41){$\beta^2\alpha^{-3}(a)$}
\put(292,79){$\beta(a)$}
\put(364,41){$a$}
\put(346,79){$\beta^3\alpha^{-3}(a)$}

\end{overpic}
\caption{The oriented multi-virtual knot $K(3)$ colored with an operator quandle $(Q,*,(\alpha,\beta)).$}\label{Fg:K3}
\end{figure}

\begin{prop}
Let $p<q$ be odd primes. Then $K(p)$ is not equivalent to $K(q)$.
\end{prop}
\begin{proof}
Attempting to color the multi-virtual knot $K(n)$ with elements of an operator quandle $(Q,*,(\alpha,\beta))$ yields the condition
\begin{equation}\label{Eq:K3Condition}
    a*\beta^n\alpha^{-n}(a)=a,
\end{equation}
as in Figure \ref{Fg:K3}. Consider the dihedral quandle $(Q,*) = (\mathbb Z_p,*)$ with multiplication $x*y = 2y-x$ and note that $\rho(x)=x+1$ is an automorphism of $(Q,*)$ of order $p$. Equip $Q$ with the operators $\alpha=1$ and $\beta=\rho$. The condition \eqref{Eq:K3Condition} for $K(p)$ then becomes $a*\rho^p(a)=a$, which holds for every $a\in Q$ since $\rho^p(a)=a$ and $a*a=a$. Hence there are $p$ colorings of $K(p)$ by the elements of the operator quandle $(Q,*,(1,\rho))$. Keeping the operator quandle but changing the knot, the condition \eqref{Eq:K3Condition} for $K(q)$ becomes $a*\rho^q(a)=a$. Now, $a*\rho^q(a) = a*(a+q) = 2(a+q)-a = a+2q\not\equiv a \pmod p$ since $p<q$ are odd primes. Hence there are no colorings of $K(q)$ by the elements of the operator quandle $(Q,*,(1,\rho))$. By Theorem \ref{Th:Coloring}, the multi-virtual knots $K(p)$ and $K(q)$ are not equivalent.
\end{proof}

Note that $\mathrm{cr}_t(K(n))\le 2n+1$ but we we do not claim that equality holds. Nevertheless, Corollary \ref{Cr:MVBound} guarantees that the sequence $(\mathrm{cr}_t(K(p)):p\text{ a prime })$ is not bounded.

\begin{problem}
Let $K(n)$ be the multi-virtual knot constructed above. Is the total crossing number of $K(n)$ equal to $2n+1$?
\end{problem}

\begin{problem}
Initiate the catalog of multi-virtual knots organized by their total crossing number. 
\end{problem}

The discussion in Subsection \ref{Ss:SameVirtual} implies that for a multi-virtual knot with $n$ virtual crossings, it suffices to consider at most $2n$ types of virtual crossings in order to fully understand the equivalence problem. The catalog of multi-virtual knots can therefore be organized by the total crossing number thanks to Proposition \ref{Pr:MVBound}.

\section{Appendix: Multi-virtual versions of the first eight virtual knots}\label{Sc:Diagrams}

\begin{figure}[!ht]
\centering
\begin{overpic}[width = 0.32\textwidth]{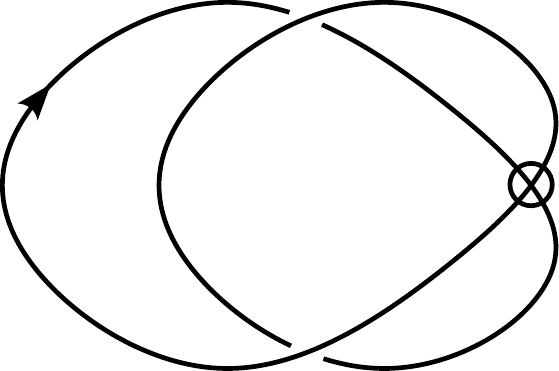}
\put(152,48){$\alpha$}
\put(33,48){$a$}
\put(-35,48){$\alpha^{-1}(a)$}
\put(75,65){$\alpha^{-1}(a)/a$}
\put(145,10){$a*\alpha(a)$}
\put(180,48){$\color{red}\bullet$\ $(a/\alpha(a))/\alpha^{-1}(a)=a$}
\put(80,11){$\color{red}\bullet$}
\end{overpic}
\caption{The virtual knot $2_1$ and its operator quandle coloring condition.}\label{Fg:21}
\end{figure}

\begin{figure}[!ht]
\centering
\begin{overpic}[width = 0.25\textwidth]{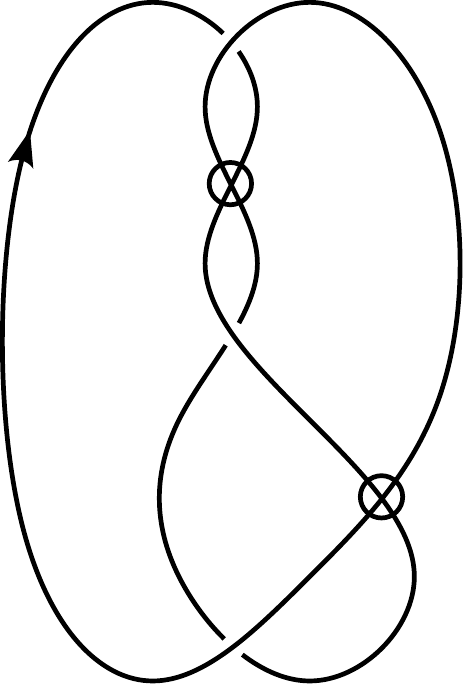}
\put(105,45){$\alpha$}
\put(66,124){$\beta$}
\put(70,105){$a$}
\put(108,160){$\beta^{-1}(a)$}
\put(-45,20){$\alpha^{-1}\beta^{-1}(a)$}
\put(70,145){$x$}
\put(40,105){$y$}
\put(105,10){$a/\alpha(a)$}
\put(30,40){$z$}
\put(56,97){$\color{red}\bullet$}
\put(165,85){$\color{red}\bullet$\ $y*z=a$}
\end{overpic}
\caption{The virtual knot $3_1$ and its operator quandle coloring condition. Here, $x=\alpha^{-1}\beta^{-1}(a)/\beta^{-1}(a)$, $y=\alpha^{-1}(a)/a$ and $z=(a/\alpha(a))/\alpha^{-1}\beta^{-1}(a)$.}\label{Fg:31}
\end{figure}

\begin{figure}[!ht]
\centering
\begin{overpic}[width = 0.25\textwidth]{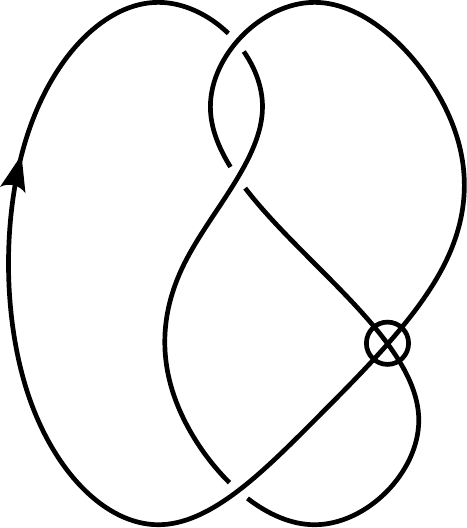}
\put(106,44){$\alpha$}
\put(42,105){$a$}
\put(-23,70){$\alpha(a)$}
\put(68,95){$\alpha(a)/a$}
\put(108,10){$(\alpha(a)/a)*\alpha(a)$}
\put(70,60){$x$}
\put(57,95){$\color{red}\bullet$}
\put(160,68){$\color{red}\bullet$\ $x/(\alpha(a)/a)=a$}
\end{overpic}
\caption{The virtual knot $3_2$ and its operator quandle coloring condition. Here, $x=(a/\alpha^{-1}(a))*a$.}\label{Fg:32}
\end{figure}

\begin{figure}[!ht]
\centering
\begin{overpic}[width = 0.25\textwidth]{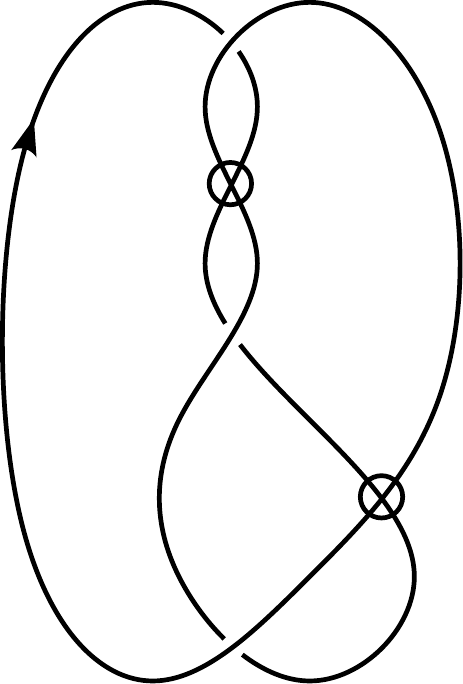}
\put(105,45){$\alpha$}
\put(68,124){$\beta$}
\put(30,40){$a$}
\put(108,160){$\beta^{-1}(a)$}
\put(-53,85){$\alpha^{-1}\beta^{-1}(a)$}
\put(70,148){$x$}
\put(8,95){$\alpha^{-1}(a){/}a$}
\put(85,70){$y$}
\put(105,10){$z$}
\put(57,15){$\color{red}\bullet$}
\put(160,85){$\color{red}\bullet$\ $z/\alpha^{-1}\beta^{-1}(a)=a$}
\end{overpic}
\caption{The virtual knot $3_3$ and its operator quandle coloring condition. Here, $x=\alpha^{-1}\beta^{-1}(a)/\beta^{-1}(a)$, $y=(\alpha^{-1}(a)/a)/a$ and $z=(a/\alpha(a))/\alpha(a)$.}\label{Fg:33}
\end{figure}

\begin{figure}[!ht]
\centering
\begin{overpic}[width = 0.25\textwidth]{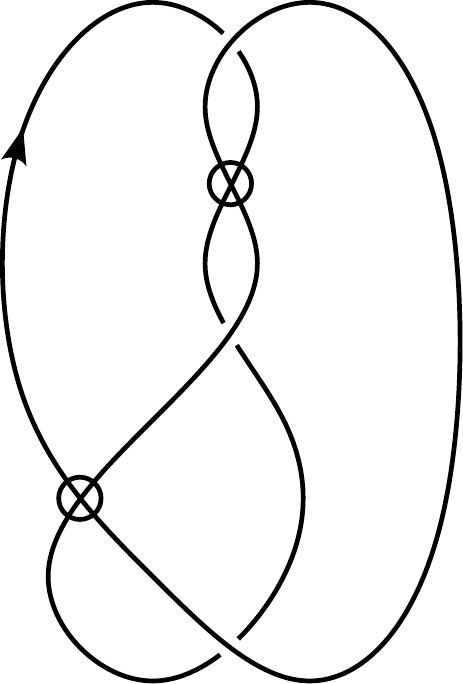}
\put(28,45){$\alpha$}
\put(67,124){$\beta$}
\put(8,10){$a$}
\put(69,106){$\alpha^{-1}(a)$}
\put(105,160){$\beta^{-1}\alpha^{-1}(a)$}
\put(-34,85){$\beta^{-1}(a)$}
\put(70,145){$x$}
\put(7,98){$a/\alpha^{-1}(a)$}
\put(82,50){$y$}
\put(54,14){$\color{red}\bullet$}
\put(160,85){$\color{red}\bullet$\ $y*\beta^{-1}\alpha^{-1}(a)=a$}
\end{overpic}
\caption{The virtual knot $3_4$ and its operator quandle coloring condition. Here, $x=\beta^{-1}(a)/\beta^{-1}\alpha^{-1}(a)$ and $y=(a/\alpha^{-1}(a))/\alpha^{-1}(a)$.}\label{Fg:34}
\end{figure}

\begin{figure}[!ht]
\centering
\begin{overpic}[width = 0.25\textwidth]{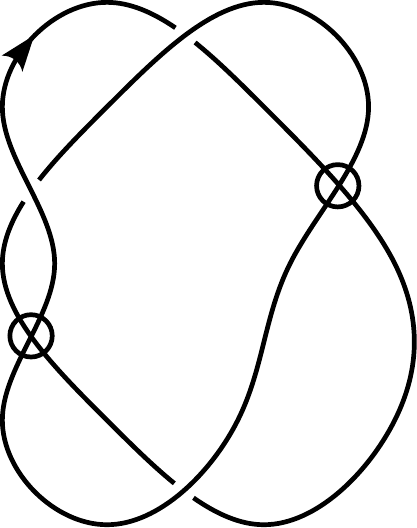}
\put(18,52){$\alpha$}
\put(104,95){$\beta$}
\put(25,123){$a$}
\put(-30,10){$\beta^{-1}(a)$}
\put(-52,130){$\alpha^{-1}\beta^{-1}(a)$}
\put(75,125){$x$}
\put(108,10){$\alpha^{-1}(a)/\beta(a)$}
\put(33,35){$y$}
\put(-10,75){$z$}
\put(7,103){$\color{red}\bullet$}
\put(160,75){$\color{red}\bullet$\ $z/\alpha^{-1}\beta^{-1}(a)=a$}
\end{overpic}
\caption{The virtual knot $3_5$ and its operator quandle coloring condition. Here, $x=\alpha^{-1}\beta^{-1}(a)/a$, $y=(\alpha^{-1}(a)/\beta(a))/\beta^{-1}(a)$ and $z=(a/\alpha\beta(a))/\alpha\beta^{-1}(a)$.}\label{Fg:35}
\end{figure}

\begin{figure}[!ht]
\centering
\begin{overpic}[width = 0.32\textwidth]{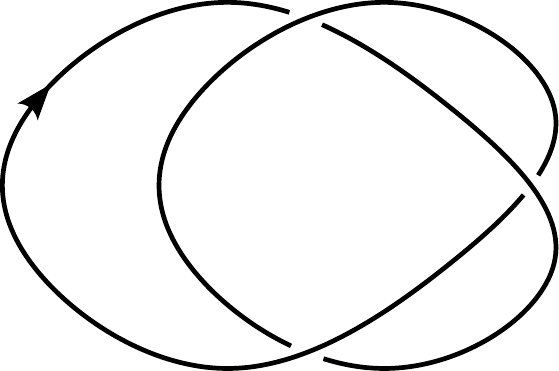}
\put(108,68){$a$}
\put(50,50){$b$}
\put(-20,50){$b/a$}
\put(80,85){$\color{red}\bullet$}
\put(80,10){$\color{red}\circ$}
\put(180,60){$\color{red}\bullet$\ $(b/a)/b=a$}
\put(180,30){$\color{red}\circ$\ $a/(b/a)=b$}
\end{overpic}
\caption{The virtual knot $3_6$ (more commonly known as the left-handed trefoil knot) and its (operator) quandle coloring conditions.}\label{Fg:36}
\end{figure}
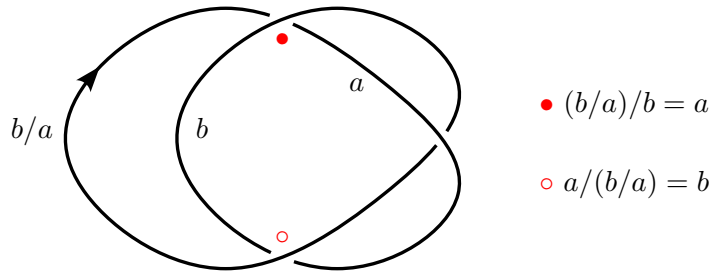

\clearpage

\begin{figure}[!ht]
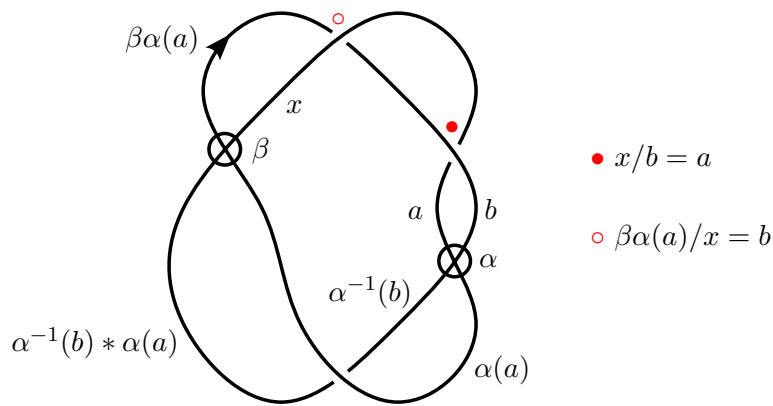

\centering
\begin{overpic}[width = 0.25\textwidth]{37.pdf}
\put(118,52){$\alpha$}
\put(32,93){$\beta$}
\put(91,70){$a$}
\put(120,70){$b$}
\put(116,10){$\alpha(a)$}
\put(61,39){$\alpha^{-1}(b)$}
\put(-59,20){$\alpha^{-1}(b)*\alpha(a)$}
\put(45,110){$x$}
\put(-16,135){$\beta\alpha(a)$}
\put(105,102){$\color{red}\bullet$}
\put(62,143){$\color{red}\circ$}
\put(160,90){$\color{red}\bullet$\ $x/b=a$}
\put(160,60){$\color{red}\circ$\ $\beta\alpha(a)/x = b$}
\end{overpic}
\caption{The virtual knot $3_7$ and its operator quandle coloring conditions. Here, $x=\beta^{-1}(\alpha^{-1}(b)*\alpha(a))$.}\label{Fg:37}
\end{figure}

\end{document}